\documentclass[12pt]{article} 
\usepackage{amsmath,amssymb,amsthm,bbm,geometry,amsfonts,mathrsfs}
\title{Partial Regularity for Harmonic Maps into Spheres at a Singular or Degenerate Free Boundary}
\date{}

\author{Roger Moser \thanks{Department of Mathematical Sciences, University of Bath, Bath BA2 7AY, UK. Email: r.moser@bath.ac.uk } \and James Roberts \thanks{Department of Mathematical Sciences, University of Bath, Bath BA2 7AY, UK. Email: j.e.roberts@bath.ac.uk}}

\numberwithin{equation}{section}
\newtheorem{thm}{Theorem}

\newtheorem{lem}{Lemma}
\newtheorem{cor}{Corollary}
\theoremstyle{definition}

\theoremstyle{remark}
\newtheorem*{rem}{Remark}
\newtheorem*{conv}{Convention}

\numberwithin{thm}{section}
\numberwithin{prop}{section}
\numberwithin{lem}{section}
\numberwithin{defn}{section}
\numberwithin{cor}{section}

\begin{document}

\maketitle
 
 \begin{abstract}
 We prove partial regularity of weakly stationary harmonic maps with (partially) free boundary data on manifolds where the domain metric may degenerate or become singular along the free boundary at the rate $d^\alpha$ for the distance function $d$ from the boundary.
 \end{abstract}

 \section{Introduction}\label{intro}
 
The regularity of harmonic mappings of Riemannian manifolds with free or constrained boundary data has been increasingly analysed in recent years in view of the connection between these maps and recent definitions of fractional harmonic mappings of Riemannian manifolds, see for example \cite{millot2019partial,millot2017minimizing,moser2011intrinsic,roberts2018regularity}. In order to obtain (partial) regularity for fractional harmonic maps, one is lead to the study of free boundary harmonic maps on domains where the Riemannian metric may degenerate or become singular along part of the boundary of the domain, depending on the fractional power in question.  The regularity of free boundary harmonic maps, with smooth bounded metrics on the domain, was studied substantially earlier and is of independent interest. For example,  free boundary harmonic maps constitute generalisations of minimal surfaces with free boundary. As pointed out in  \cite{MR918841}, they are also connected with critical points of partially linearised models of elasticity with free boundary conditions and can also arise in the theory of liquid crystals when modelling surface effects near a solid boundary interface. 
Motivated by their connection with fractional harmonic maps and the possible other geometric and physical applications, our goal is to establish a partial regularity theory for a class of free boundary harmonic mappings on domains where the metric is
conformal to a smooth Riemannian background metric with conformal factor blowing up or vanishing at the boundary.
We state some of our results for general target manifolds, but our main theorems concern sphere valued maps.

Let $\mathcal{M}$ be a smooth Riemannian manifold of dimension  $m \geq 3$ with smooth
non-empty boundary, equipped with a smooth, bounded Riemannian metric $g$. Suppose $\mathcal{N}$ is a smooth, compact Riemannian manifold which we may assume is isometrically embedded in $\mathbb{R}^n$ for some $n \in \mathbb{N}$ due to the theorem of Nash \cite{MR0075639}.  Let $d := \text{dist}(\cdot,\partial\mathcal{M})$. We are particularly interested in Riemannian metrics $h$ on $\mathcal{M}$ of the form
\begin{equation*}
h = d^\alpha g 
\end{equation*}
in a neighbourhood of $\partial \mathcal{M}$
for a fixed $\alpha \in (-\frac{2}{m-2}, \frac{2}{m-2})$;  we will  analyse the regularity of a class of critical points  $v:(\mathcal{M},h)\to\mathbb{R}^n$ of the Dirichlet energy on $(\mathcal{M},h)$  given by
\begin{equation*}
E_{\beta}(v) = \frac{1}{2}\int_{\mathcal{M}} \vert Dv\vert_{h}^2\mathrm{dvol}_{h} = \frac{1}{2}\int_{\mathcal{M}} d^{\beta}\vert Dv\vert_{g}^2\mathrm{dvol}_{g}.
\end{equation*}
Here $D$ is the differential of $v$, $\beta = \frac{\alpha (m-2)}{2} \in (-1,1)$, and $\vert Dv\vert_{g}^2$ is the energy density with respect to $g$. In local coordinates it is written as  
$\vert Dv\vert_{g}^2 = g^{ij}\langle \partial_iv,\partial_jv\rangle$ where $\langle \cdot,\cdot\rangle$ is the Euclidean metric and $(g^{ij}) = (g_{ij})^{-1}$ is the matrix representing $g^{-1}$, and $\mathrm{dvol}_{g}$ is the volume form which in local coordinates satisfies $\mathrm{dvol}_{g} = \sqrt{\text{det}(g)}\mathrm{d}x $.
Throughout the paper, we assume the convention that repeated indices indicate summation over the appropriate range (unless
the index is $m$, which is fixed).

Observe that when $\beta = 0$ the energy $E_{0}$ is nothing other than the Dirichlet energy on the manifold $(\mathcal{M},g)$. In particular, if we were to allow $m = 2$ then, regardless of the power $\alpha$, we have $E_{\beta}\equiv E_0$; in this case all the results in this article follow from known results for free boundary harmonic maps, which we will discuss subsequently. 

In \cite{roberts2018regularity} the second author established partial regularity up to the boundary for minimisers $v:\mathcal{M}\to\mathcal{N}$  of $E_{\beta}$ with respect to a free boundary condition in the case where $\mathcal{M}$ is a Euclidean half-space. In this article  we will establish partial regularity for maps which are merely critical points of $E_{\beta}$ with respect to both outer and inner variations
and still with a free boundary condition. Henceforth we refer to such maps as called \emph{weakly stationary harmonic maps with free boundary data}.

If we have free boundary data on the whole natural boundary of a manifold, we may find that only constant solutions to the
problem exist. As we are studying only the free boundary here, however, we simply exclude any part
where the boundary data are not free. Thus we may imagine that $\mathcal{M}$ is part of a larger manifold $\mathcal{M}'$
and $v$ is prescribed on $\mathcal{M}' \setminus \mathcal{M}$, but everywhere on $\partial \mathcal{M}$, it is free.
As a consequence, we do not assume that $\mathcal{M}$ is complete.

For reasons discussed subsequently, we will also focus on the case where $\mathcal{N} = \mathbb{S}^{n-1}\subset\mathbb{R}^n$ is the round unit sphere.  An abridged version of our partial regularity theorem states the following.

 \begin{thm}\label{mainthmabridged}
Suppose
$v \colon \mathcal{M} \to \mathbb{S}^{n-1}$ is weakly stationary harmonic with free boundary data.
Then there is a closed set $\Sigma \subset \mathcal{M}$ such that $v$ is H\"older continuous in $\mathcal{M}\backslash\Sigma$
and $\mathscr{H}^{m-2}(\Sigma\cap \mathrm{int}(\mathcal{M})) = 0$ and $\mathscr{H}^{m+\beta-2}(\Sigma\cap\partial \mathcal{M}) = 0$. 
\end{thm}
\noindent Here, $\text{int}(\mathcal{M}) = \mathcal{M} \setminus \partial \mathcal{M}$ is the interior of $\mathcal{M}$
and $\mathscr{H}^t$ is the $t$-dimensional Hausdorff measure. The theorem should properly be formulated for maps in
a weighted Sobolev space, with weights behaving like $d^\beta$ near the boundary, but we leave the technical details for later.
  
Away from the boundary $\partial \mathcal{M}$, the regularity stated in the Theorem \ref{mainthmabridged} follows from known regularity theory for \emph{stationary harmonic maps}; these maps are critical points, with respect to both outer and inner variations of the Dirichlet energy for mappings between Riemannian manifolds.  H\'elein established that weakly harmonic maps on two-dimensional domains are smooth \cite{MR1131583}. However, when the dimension of the domain $m$ satisfies $m \geq 3$, harmonic maps are known to have singularities (points of discontinuity) in general and may even be discontinuous on the whole of their domain;  Rivi\`ere has constructed discontinuous harmonic maps into spheres in \cite{MR1368247}. In order to achieve even partial regularity, additional constraints on harmonic maps are required. One way to proceed is to impose geometric constraints, such as the co-domain having non-positive sectional curvature. In this case, harmonic maps are also known to be smooth, see  \cite{MR664498} for example.  In order to deal with general codomain manifolds however, one must impose a different type of constraint. Typically one considers  stationary harmonic maps, or at least maps which satisfy a monotonicity inequality for the re-scaled energy.     For such maps it is possible to estimate the size of the singular set; stationary harmonic maps are smooth in the interior of their domain with  the possible exception of a set of vanishing $(m-2)$-dimensional Hausdorff measure. This was first established for sphere valued maps by Evans \cite{MR1143435} and then extended to more general codomain manifolds $\mathcal{N}$ by Bethuel \cite{MR1208652}. 

Near the boundary, when $\beta = 0$  the regularity stated in Theorem \ref{mainthmabridged} follows from the existing regularity theory for stationary free boundary harmonic maps. As mentioned previously, the regularity theory for such maps was originally investigated with geometric or physical motivations in mind. The regularity of free boundary harmonic maps has been investigated by Baldes \cite{MR683042} and Gulliver and Jost \cite{MR918841} who established full regularity of free boundary harmonic maps under certain geometric constraints, such as the image of the harmonic map lies in a sufficiently small geodesic ball. Scheven later proved partial regularity of free boundary stationary harmonic maps into general target manifolds   \cite{MR2206640} using a reflection  construction in the codomain.
  
When $\beta \neq 0$, Theorem \ref{mainthmabridged} is new. There are related results
of the second author \cite{roberts2018regularity} (studying minimisers of $E_\beta$)
and of Millot, Pegon and Schikorra \cite{millot2019partial} who consider a class of stationary or minimising critical points $v:\mathbb{R}^{m-1}\times (0,\infty)\to\mathbb{R}^n$ of $E_{\beta}$, where the domain is the half-space $\mathbb{R}^{m-1}\times (0,\infty)$ and $v$ satisfies the boundary constraint $v(\mathcal{O})\subset\mathbb{S}^{n-1}$ for some open $\mathcal{O}\subset\partial(\mathbb{R}^{m-1}\times (0,\infty))$. Millot, Pegon and Schikorra use the maps they considered as means to establish partial regularity of a type of  fractional $s$-harmonic maps into spheres. The maps they consider satisfy $\text{div}(x_m^{\beta}Dv) = 0$ in a half-space $\mathbb{R}^{m-1}\times (0,\infty)$ together with  a non-linear Neumann boundary condition. In contrast we require the maps we consider to satisfy   $v(\mathcal{M}) \subset \mathcal{N}$,
which makes the condition   $v(\mathcal{O})\subset \mathcal{N}$ redundant.   Such maps satisfy a harmonic map equation in the interior of the domain and a homogeneous Neumann-type condition on the boundary, see Section \ref{firstvar}.
  
We will prove Theorem  \ref{mainthmabridged} by establishing an  $\varepsilon$-regularity result, see Theorem \ref{epregthm}, and then use a covering argument to conclude the partial regularity stated in Theorem \ref{mainthmabridged}. 
 The key estimate required to prove Theorem \ref{mainthmabridged} is a Caccioppoli-type inequality, see Lemma \ref{cacctypeineq}. For minimisers this kind of inequality is established using comparison maps and we used this approach to prove partial regularity for minimisers of $E_{\beta}$ in \cite{roberts2018regularity}. When considering stationary harmonic maps, it is no longer possible to take advantage of the minimising property and a different approach is required; we will instead take advantage of compensated compactness phenomena in the Euler-Lagrange equation. 

The underlying observation is that for harmonic maps into a sphere, the expression $|Dv|^2$ can be written as a product with a
`div-curl' structure. This idea goes back to H\'elein, who used it to show regularity of weakly harmonic maps from surfaces into spheres \cite{helein1990regularite}. Evans \cite{MR1143435} extended the results of H\'elein to higher dimensional domains.
The availability of such methods is the reason why we concentrate on sphere-valued maps as well.
The regularity theory for the usual harmonic maps was further extended to other target manifolds by Bethuel \cite{MR1208652}, also
using concentrated compactness arguments but with different technical details. There is also an alternative approach
due to Rivi\`ere \cite{Riviere2007}. There is no reason \emph{a priori} why these methods may not be adapted to
our setting, too. However, as we wish to avoid the technical difficulties that come with greater generality and
concentrate on the new challenges stemming from a Riemannian metric of the form $d^\alpha g$, we mostly work with the sphere here.

One of the cornerstones of the regularity theory for harmonic maps is the duality between BMO-spaces and the Hardy
spaces $\mathcal{H}^1(\mathbb{R}^m)$, combined with estimates for the corresponding norms (see, e.g., the work of
Coifman-Lions-Meyer-Semmes \cite{coifman1990compacite} or of Chanillo \cite{doi:10.1080/03605309108820829},
who bypasses this duality but obtains inequalities in the same spirit nevertheless). We will require similar
estimates, but adapted to weighted Sobolev spaces, which do not appear in the existing literature.
The derivation of suitable inequalities will
therefore occupy a significant portion of this paper.

In the aforementioned regularity theories for harmonic, and fractional harmonic, maps it is known that stationary or minimising maps are smooth away from their singular set. In the context where these regularity theories were developed, higher regularity follows from known estimates for continuous or H\"older continuous solutions to semi-linear elliptic equations, see respectively \cite{jost2008riemannian}  and \cite{MR765241} for example. In \cite{roberts2018regularity} the second author studied minimising free boundary harmonic maps
in the case where $(\mathcal{M}, g)$ is a half-space with the Euclidean metric. He used related ideas and some additional arguments to show
smoothness in directions tangential to free boundary despite the singular factor arising from the factor $d^{\beta}$ in the energy. In this article, the anisotropy of the metric $g$ precludes us from directly applying the theory in \cite{roberts2018regularity} unless $g$ is Euclidean (in which case we may may draw conclusions on the smoothness of the maps we consider in directions tangential to the boundary as well). Once H\"older continuity is known, we expect minor modifications of the higher regularity results in \cite{roberts2018regularity} to yield smoothness for the stationary free boundary harmonic maps we consider here. In particular, we expect all the conclusions of Theorem 4.3 of \cite{roberts2018regularity} pertaining to the regularity of the maps considered there, to hold for the maps we consider here (with the exception of the reduced dimension of the singular set in the interior of the domain). However, we will not give the details since we expect all the necessary modifications to be of a purely technical nature. We furthermore observe that there are emerging regularity theories for semi-linear degenerate elliptic equations of the form we consider here, see for example \cite{sire2020liouville}. We note that we could prove a version of Lemma 4.23 in \cite{roberts2018regularity} showing H\"older continuity implies Lipschitz continuity and then the results of \cite{sire2020liouville} would apply to give  $C^{1,\gamma}$ regularity and possibly even $C^{\infty}$ regularity depending on the precise assumptions required.  For the aforementioned reasons, we will not discuss higher regularity any further in this article. 

We now outline the organisation of the paper and the strategy for the proof, further details and references may be found at the beginning of each section. In Section \ref{metriccoord} we introduce local coordinates such that the factor $d^{\beta}$ becomes the weight $x_m^{\beta}$ near the boundary; we also introduce the function spaces we require.  As discussed previously, the range of specified $\alpha$ implies $\beta \in (-1,1)$. It is then well known that the weight $x_m^{\beta}$ then lies in Muckenhoupt class $A_2$, see Section \ref{hardy} for the definition, which allows us to apply weighted counterparts to many results, such as the Sobolev-Poincar\'e inequality, which hold for unweighted Sobolev spaces. In Section \ref{firstvar} we give the full definition of stationary harmonic maps with free boundary data and state the corresponding Euler-Lagrange equations. In Section \ref{mon} we use the criticality with respect to inner variations to establish energy monotonicity formulas for the re-scaled energy on concentric balls for maps into general target manifolds. We first do this in the coordinates in Section \ref{metriccoord} and then show that monotonicity inequalities hold independently of the choice of coordinates. We prove control of the $L^2$ norm of the average difference between $v$ and its average provided the re-scaled energy is assumed sufficiently small in Section \ref{l2}, again for general target manifolds. From Section \ref{spheres} onwards we specialise to the case where the target manifold is a Euclidean sphere. In this case the Euler-Lagrange equations have a similar structure to the unweighted case and we show that $d^{\beta}\vert Dv\vert^2$ can be written in the form of a weighted `div-curl' structure. In Section \ref{hardy} we establish weighted versions of the aforementioned compensated compactness results, both globally and locally. We then use the weighted compensated compactness arguments, together with the structure of the Euler-Lagrange equations to  establish a Caccioppoli-type inequality for the maps we consider in Section \ref{caccdecay}. With this in hand, it is possible to prove that the re-scaled energy at the boundary decays faster than obtained in the monotonicity formulas described in Section \ref{mon}. We combine this fact with known results in the interior for stationary harmonic maps to obtain $\varepsilon$-regularity and partial regularity in Section \ref{epsilonreg}.

\begin{conv}
We introduce the convention that unless stated otherwise we write $C$ to denote a constant which depends only on universal factors such as $m,\beta$ or other uniform constants introduced in the next section and we will not always distinguish between different such constants. 
\end{conv}

\section{Local Coordinates and Function Spaces}\label{metriccoord}

Many of our arguments will make use of local coordinates. As we primarily focus on an analysis of critical points of $E_{\beta}$ near part of the boundary $\partial\mathcal{M}$, it will be convenient to choose local coordinates that
 map a piece of $\partial \mathcal{M}$ to $\mathbb{R}^{m - 1} \times \{0\}$ and at the same time
 preserve the distance (with respect to $g$) to the boundary.  
 
 Consider a coordinate chart $x' \colon \mathcal{U}' \to \mathbb{R}^{m - 1}$ on $\partial \mathcal{M}$ centred at $x_0 := x(p)\in \mathbb{R}^{m-1}\times\{0\}$, where $\mathcal{U}' \subset \partial \mathcal{M}$ is
 a relatively open set.  We may choose an open set $\mathcal{U} \subset \mathcal{M}$ such that $\mathcal{U}' = \mathcal{U} \cap \partial \mathcal{M}$ and such that
the nearest point projection $\pi_{\partial \mathcal{M}} \colon \mathcal{U} \to \partial \mathcal{M}$ with respect to $g$ is well-defined and smooth in $\mathcal{U}$.
Then the distance function $d$ is smooth as well. We define $x \colon \mathcal{U} \to \mathbb{R}^m$ by
\[
x(q) = (x'(\pi_{\partial M}(q)), d(q)), \quad q \in \mathcal{U}
\]
i.e., $x = (x' \circ \pi_{\partial \mathcal{M}}) \times d$. This gives rise to local coordinates $x$ on $\mathcal{U}$.
If we set $\mathscr{U} = x(\mathcal{U})$, then for $x' \in \mathbb{R}^{m - 1}$, the lines
$(\{x'\} \times \mathbb{R}) \cap \mathscr{U}$ correspond to geodesics in $\mathcal{U}$ intersecting the boundary
perpendicularly, and distances with respect to $g$ are preserved
by $x$ along these geodesics. If $(g_{ij})$ and $(h_{ij})$ denote the components of $g$ and $h$, respectively, in these
local coordinates, then
\begin{equation}\label{metricbdrylocal}
h_{ij}(x) = x_m^\alpha g_{ij}(x), \quad i, j = 1, \ldots, m.
\end{equation}
 	
Some of our arguments do not require anything more. However, we will sometimes need to take advantage of some additional
properties satisfied by these coordinates. For example, we will prove that a variant of the well-known monotonicity formula for harmonic maps near points of the boundary holds in a uniform way.

Note that the background metric $g$ induces a Riemannian metric $g'$ on $\partial \mathcal{M}$ by restriction (even though $h$ does not).
Consider a point $p \in \partial \mathcal{M}$ and suppose that $x' \colon \mathcal{U}' \to \mathbb{R}^{m - 1}$ describes \emph{normal} coordinates
on $\partial \mathcal{M}$ about $p$ with respect to $g'$. 
If we use the above construction to obtain local coordinates $x \colon \mathcal{U} \to \mathbb{R}^m$ with this choice
of $x'$, then we automatically find that
\[
g_{ij}(x_0) = \delta_{ij}.
\]
As everything is smooth, it follows that there exist $\rho_p > 0$ and $c_p > 0$ such that
\begin{equation}\label{coordmetricbounds}
|g_{ij}(x) - \delta_{ij}| \le c_p|x| \quad \text{and} \quad |\partial_k g_{ij}| \le c_p \quad \text{for } i, j, k = 1, \dotsc, m,
\end{equation}
as long as $|x- x_0| \le \rho_p$. Moreover, if $B_r^g(p)$ denotes the ball of radius $r$ about $p$ with respect to the
metric $g$, then for any $r < \rho_p$,
\begin{equation}\label{comparisonofballs}
B_{r - c_pr^2}^g(p) \subset \{q \in \mathcal{M}\colon |x(q)-x_0| < r\} \subset B_{r + c_pr^2}^g(p).
\end{equation}
This gives a useful comparison between balls with respect to $g$ on the one hand
and the preimages of Euclidean balls with respect to $x$ on the other hand.

The functions $p \mapsto \rho_p$ and $p \mapsto c_p$ may be chosen continuous. Hence for any compact
set $\mathcal{K} \subset \mathcal{U}$, there exist $\rho > 0$ and $c > 0$ such that $\rho_p \ge \rho$ and $c_p \le c$
for all $p \in \mathcal{K} \cap \partial \mathcal{M}$. The above inequalities are then uniform in $\mathcal{K}$.
Similarly, there exist $C_0, \dotsc, C_3$ such that the inequalities
\begin{equation}\label{metricchartbounds}
C_0\vert \xi\vert^2 \leq g_{ij}\xi_i\xi_j, \quad g^{ij}\xi_i\xi_j\leq C_1\vert \xi\vert^2,  \quad \vert g_{ij}\vert,\vert g^{ij}\vert \leq C_2,
\end{equation}
for every $\xi \in \mathbb{R}^m$, and
\begin{equation}\label{metricchartderivbounds}
\vert \partial_kg_{ij}\vert,\vert \partial_kg^{ij}\vert\leq C_3 
\end{equation}
for $k = 1,\ldots,m$, hold uniformly in $\mathcal{K}$ even if $x$ is not based on normal coordinates.

We are interested in the regularity of a class of critical points $v:\mathcal{M}\to\mathcal{N}$ of $E_{\beta}$  such that the energy is, at least locally, finite. We first observe that in the coordinates we have introduced on $\mathscr{U}$ the energy takes the form  
\begin{equation*}
E_{\beta}\vert_{\mathscr{U}}(v) := \frac{1}{2} \int_{\mathcal{U}} d^\beta \vert Dv\vert^2_{g}\mathrm{dvol}_{g} = \frac{1}{2}\int_{\mathscr{U}} x_m^{\beta} g^{ij} \langle \partial_i v, \partial_j v\rangle \sqrt{\det(g)} \mathrm{d}x.
\end{equation*}
In order to give meaning to the notion of critical points of $E_{\beta}$ we consider  Sobolev spaces on $\mathcal{M}$ such that this expression is, at least locally, finite. We therefore consider function spaces with respect to the weight $x_m^{\beta}$ for $\beta \in (-1,1)$. 
For open $\mathscr{U}\subset\mathbb{R}^m$ we denote $L^p_{\beta}(\mathscr{U};\mathbb{R}^n)$ to be the Banach space of
measurable functions $f \colon \mathscr{U} \to \mathbb{R}^n$ such that the norm $\vert \vert f \vert\vert_{L^p_{\beta}(\mathscr{U};\mathbb{R}^n)}:= \left(\int_{\mathscr{U}}\vert x_m\vert^{\beta}\vert f \vert^p \right)^{\frac{1}{p}}$ is finite. Then $W^{1,2}_{\beta}(\mathscr{U};\mathbb{R}^n)$ is the Banach Space of functions $f \in L_\beta^p(\mathscr{U}; \mathbb{R}^n)$ with weak derivatives also in $L^p_{\beta}(\mathscr{U};\mathbb{R}^n)$ and with norm  $\vert \vert f \vert\vert_{W^{1,p}_{\beta}(\mathscr{U};\mathbb{R}^n)}:=\left(\int_{\mathscr{U}}\vert x_m\vert^{\beta}\vert f \vert^p  + \int_{\mathscr{U}}\vert x_m\vert^{\beta}\vert Df \vert^p\right)^{\frac{1}{p}}$. When $p = 2$, the spaces $L^2_{\beta}$ and $W^{1,2}_{\beta}$ are Hilbert Spaces with inner products which induce the respective norms. We further define $W^{1,p}_{\beta,0}(\mathscr{U};\mathbb{R}^n)$ as the closure in $W^{1,p}_{\beta}(\mathscr{U};\mathbb{R}^n)$ of the space of smooth compactly supported functions $C_0^{\infty}(\mathscr{U};\mathbb{R}^n)$ on $\mathscr{U}$ with respect to   $\vert \vert  \cdot\vert\vert_{W^{1,p}_{\beta}(\mathscr{U};\mathbb{R}^n)}$. Since $\vert x_m\vert^{\beta}$ is an $A_2$ weight, as defined subsequently in  \eqref{Apcond}, the space of smooth functions in $W^{1,p}_{\beta}(\mathscr{U};\mathbb{R}^n)$ is dense in $W^{1,p}_{\beta}(\mathscr{U};\mathbb{R}^n)$ for $p\geq 2$, and even for $p >q\geq 1$ for some $2>q \geq 1$, but not necessarily all $p \in [1,\infty)$, see \cite{MR1774162} Corollary 2.1.6.  When $\beta = 0$, we omit the subscript in the preceding notation.

Since we are interested in manifold valued maps, we now introduce the function spaces we will use to analyse manifold valued critical points of $E_{\beta}$. We define $W^{1,2}_{\beta}(\mathscr{U};\mathcal{N})$ as the collection of maps in $W^{1,2}_{\beta}(\mathscr{U};\mathbb{R}^n)$ with values in $\mathcal{N}$ for Lebesgue almost every $x \in \mathscr{U}$. We may then define $W^{1,2}_{\beta,\text{loc}}(\mathcal{M};\mathcal{N})$
as maps such that $v \circ x^{-1} \in W_\beta^{1,2}(\mathscr{K};\mathcal{N})$ for any local coordinates $x \colon \mathcal{U} \to \mathscr{U}$
as above and any compact set $\mathscr{K} \subset \mathscr{U}$, as well as
$v \in W^{1,2}(\mathcal{K};\mathcal{N})$ for every compact set $\mathcal{K} \subset \mathrm{int}(\mathcal{M})$.

We will have occasion to consider a reflection of $v \in W^{1,2}_{\beta}(\mathscr{U};\mathbb{R}^n)$ and of the metric $g$
in the hyperplane $\mathbb{R}^{m-1}$. Let $\mathscr{U}_{-}:=\{x \in \mathbb{R}^m: (x',-x_m)\in \mathscr{U}, x_m<0\}  $ and define $\mathscr{V} = \mathscr{U} \cup \mathscr{U}_{-}$.
The even reflection of $v$ in the hyperplane $\mathbb{R}^{m-1}\times\{0\}$, which we denote by $\tilde{v}$, satisfies $\tilde{v}(x) = v(x',\vert x_m\vert)$, where $x = (x',x_m) \in \mathbb{R}^{m-1}\times\mathbb{R}$  and $\tilde{v} \in W^{1,2}_{\beta}(\mathscr{V};\mathbb{R}^n)$. Furthermore, $\partial_i\tilde{v}(x) = \partial_iv(x',\vert x_m\vert)$ for $i \neq m$ and  $\partial_m\tilde{v}(x) = \text{sgn}(x_m)\partial_mv(x',\vert x_m\vert)$.

We also extend $g$ to $\mathscr{V}$ by the obvious reflection, giving rise to $\tilde{g}$ on $\mathscr{V}$,
and note that $\tilde{g}$ also satisfies
\eqref{metricchartbounds} for the same constants in all of $\mathscr{V}$.
It is worthy of note that in the coordinates specified previously in this section,
the coefficients $g_{im}= g_{mi}$ satisfy $g_{im}(x',0) = 0$ for $i < m$ and hence so do the coefficients
$g^{im}(x',0) = g^{mi}(x',0)$ for $i \neq m$. This means that $\tilde{g}$ and $\tilde{g}^{-1}$ are continuous, and even Lipschitz continuous in compact domains where they are defined. We further note that the eigenvalues of ${\tilde{g}}^{-1}$ are given by $\lambda_{k}(x',\vert x_m\vert)$ where $\lambda_{k}$ are eigenvalues of $g^{-1}$.

\section{The First Variation of $E_{\beta}$}\label{firstvar}

In this section we specify the class of critical points we consider. We give the definitions in general but only calculate the Euler-Lagrange equations in a chart $\mathscr{U}$ as specified in Section \ref{metriccoord} since this will suffice for the subsequent analysis. 
 
Let $\pi_{\mathcal{N}}$ denote the nearest point projection on $\mathcal{N}$ and let  $v \in W^{1,2}_{loc}(\mathcal{M};\mathcal{N})$. 
For compact $\mathcal{K} \subset \mathcal{M}$ we define 
$E_{\beta}\vert_{\mathcal{K}}(v) = \frac{1}{2}\int_{\mathcal{K}}d^{\beta}\vert Dv\vert^2_{g}\mathrm{dvol}_{g}$. Observe that $\mathcal{K}$ may intersect the boundary of $\mathcal{M}$.  We first consider outer variations of $v$ given by $v_t:= \pi_{\mathcal{N}}(v+t\phi) \in W^{1,2}_{loc}(\mathcal{M};\mathcal{N})$ where $\phi \in C_0^{\infty}(\mathcal{M};\mathbb{R}^n)$.
We say that $ v \in W^{1,2}_{loc}(\mathcal{M};\mathcal{N})$ is \emph{weakly harmonic with respect to free boundary data}
if for every $\phi \in C_0^{\infty}(\mathcal{M};\mathbb{R}^n)$ we have $
\partial_{t}\vert_{t=0}E_{\beta}\vert_{\mathrm{supp} \phi}(v_t) = 0$.

Now let $\mathscr{U}$ be as in Section \ref{metriccoord} with the given coordinates $x \colon \mathcal{U} \to \mathscr{U}$.
For $\psi \in C_0^\infty(\mathscr{U}; \mathbb{R}^n)$, the function $\phi = \psi \circ x$ gives rise to
an admissible outer variation of the form $v_t := \pi_{\mathcal{N}}(v + t\psi)$.
Following the computations of e.g \cite{simon1996theorems} Chapter 2, we take the first derivative of $E_{\beta}$ at $t = 0$
to see that in the chosen coordinates critical points of $E_{\beta}$ with respect to outer variations satisfy 
\begin{align}\label{wELeqn}
0 &=  \int_{\mathscr{U}} x_m^{\beta} g^{ij}(\langle  \partial_i\psi   ,\partial_jv\rangle -\langle   A(v)(\partial_iv,\partial_jv ),\psi\rangle)\mathrm{dvol}_{g}  .
\end{align}
This is the weak form of the equation
\begin{equation}\label{sELeqn}
\partial_i \left(x_m^{\beta}\sqrt{\text{det}(g)}g^{ij}\partial_jv\right) + x_m^{\beta}\sqrt{\text{det}(g)}g^{ij}A(v)(\partial_iv,\partial_jv) = 0
\end{equation}
in $\mathscr{U}$ with Neumann boundary condition
\begin{equation}\label{strongneum}
 x_m^{\beta} g^{mj}\partial_jv = 0
\end{equation}
on $\mathscr{U}\cap(\mathbb{R}^{m-1}\times\{0\})$.
More concisely, equation \eqref{sELeqn} can be written in the form
\[
\mathrm{div}_g\left(x_m^{\beta}\mathrm{grad}_g v\right) + x_m^{\beta} \mathrm{tr}_g A(v)(Dv,Dv) = 0,
\]
where $\mathrm{div}_g$, $\mathrm{grad}_g$, and $\mathrm{tr}_g$ are the divergence, gradient, and trace with
respect to $g$, respectively.

The following observation will be useful when we choose certain test functions for \eqref{sELeqn}.

\begin{lem} \label{test_functions}
If \eqref{wELeqn} is satisfied for all $\psi \in C_0^\infty(\mathscr{U};\mathbb{R}^n)$, then it also holds true for
all $\psi \in W_\beta^{1,2}(\mathscr{U};\mathbb{R}^n) \cap L^\infty(\mathscr{U}; \mathbb{R}^n)$ compactly supported in $\mathscr{U}$.
\end{lem}

\begin{proof}
It is well-known that $\vert x_m\vert^{\beta}$ is a Muckenhoupt weight of class $A_2$, see  \eqref{Apcond} for the definition.
By Corollary 2.1.6   in \cite{MR1774162} Chapter 2 this implies that $C_0^{\infty}(\mathscr{U};\mathbb{R}^n)$
is dense in $W^{1,2}_{\beta}(B_{\rho}(y);\mathbb{R}^n)$. Indeed, for
$\psi \in W_\beta^{1,2}(\mathscr{U};\mathbb{R}^n) \cap L^\infty(\mathscr{U}; \mathbb{R}^n)$ with compact support in
$\mathscr{U}$, a sequence
$(\psi_k)_{k \in \mathbb{N}}$ in $C_0^{\infty}(\mathscr{U};\mathbb{R}^n)$ converging to $\psi$ can be constructed by convolution
with a standard mollifying kernel. Then we also have $\psi_k \to \psi$ pointwise almost everywhere (possibly after passing to a subsequence)
and the sequence $(\psi_k)_{k \in \mathbb{N}}$  is bounded in $L^\infty(\mathscr{U}; \mathbb{R}^n)$.

Now it is clear that
\[
\int_{\mathscr{U}} x_m^{\beta} g^{ij}\langle  \partial_i\psi   ,\partial_jv\rangle \mathrm{dvol}_{g} = \lim_{k \to \infty} \int_{\mathscr{U}} x_m^{\beta} g^{ij}\langle  \partial_i\psi_k   ,\partial_jv\rangle \mathrm{dvol}_{g}.
\]
Furthermore, the convergence
\[
\int_{\mathscr{U}} x_m^{\beta} g^{ij}\langle   A(v)(\partial_iv,\partial_jv ),\psi\rangle\mathrm{dvol}_{g} = \lim_{k \to \infty} \int_{\mathscr{U}} x_m^{\beta} g^{ij}\langle   A(v)(\partial_iv,\partial_jv ),\psi_k\rangle\mathrm{dvol}_{g}
\]
follows from the Dominated Convergence Theorem. The statement of the lemma now follows immediately.
\end{proof}

We now consider inner variations of $E_{\beta}$.
Let $\mathcal{K} \subset \mathcal{M}$ be compact and let
$\Phi_{t} \colon \mathcal{M} \to \mathcal{M}$ constitute a smooth $1$-parameter family of diffeomorphisms
such that $\Phi_t(p) = p$ for all $p \in \mathcal{M} \setminus \mathcal{K}$ and all values of $t$.
Then $v_t := v \circ \Phi_t$ is an inner variation.
Note that $\mathcal{K}$ may include boundary points again.
We say that $v$ is \emph{weakly stationary harmonic with respect to free boundary data} if it is weakly harmonic with respect to free
boundary data and $\partial_t\vert_{t=0}E_{\beta}\vert_{\mathcal{K}}(v_t) = 0$ for all inner variations of the preceding form.

Consider again local coordinates $x \colon \mathcal{U} \to \mathscr{U}$ as in Section \ref{metriccoord}.
Let $\Psi_t(x) = x+ t\psi$ for a $\psi \in C_0^{\infty}(\mathscr{U};\mathbb{R}^m)$ with $\psi(x',0) \subset \mathbb{R}^{m-1}\times\{0\}$
for all $(x',0) \in \mathscr{U}$,  where $t$ is small enough to make $\Psi_t$ a diffeomorphism of $\mathscr{U}$.
We apply the above definition to $\Phi_t = x^{-1} \circ \Psi_t \circ x$.
An analysis of $\Psi_t$ similar to the standard theory for harmonic maps (see, e.g., \cite{simon1996theorems} again)
then shows that a critical point of $E_{\beta}$ with respect to inner variations satisfies
\begin{align}\label{innervar}
0 & = \int_{\mathscr{U}} x_m^{\beta} g^{ij} \left(\langle \partial_iv   ,\partial_{k}v   \rangle\partial_j\psi_{k} - \frac{1}{2}  \langle \partial_iv   ,\partial_jv  \rangle \mathrm{div}_g \psi\right) \mathrm{dvol}_g \nonumber\\
& \quad - \frac{1}{2} \int_{\mathscr{U}} x_m^{\beta} \partial_k g^{ij} \langle \partial_iv   ,\partial_jv  \rangle \mathrm{dvol}_g - \frac{\beta}{2} \int_{\mathscr{U}}\psi_m x_m^{\beta-1}g^{ij}\langle \partial_iv   ,\partial_jv  \rangle \mathrm{dvol}_g,
\end{align}
where again
\[
\mathrm{div}_g \psi = \frac{1}{\sqrt{\det(g)}} \partial_i\left(\sqrt{\det(g)} \psi_i\right)
\]
is the divergence of $\psi$ with respect to $g$.

If any of the criteria defining weakly or stationary harmonic maps with respect to free boundary data
are satisfied but with variations only supported away from the boundary, we refer to the maps in question as correspondingly weakly harmonic or weakly stationary harmonic, which is standard terminology.

\section{Monotonicity Formula}\label{mon}
Here we prove a monotonicity formula for half-balls centred on the boundary, similar to that for stationary harmonic maps, using arguments analogous to those of Gro\ss e-Brauckmann \cite{MR1184789}. Instead of using normal coordinates, we will use those discussed in Section \ref{metriccoord}. The second author proved a version of the following formula  in \cite{roberts2018regularity}, but the present situation is more general in view of the metric $g$.  Henceforth we define $B^+_R(x_0) = B_R(x_0)\cap (\mathbb{R}^{m-1}\times(0,\infty))$ where $B_R(x_0) = \{x \in \mathbb{R}^m: \vert x - x_0\vert <R\}$ is a Euclidean ball. We also write $\partial^+\Omega = \partial \Omega \cap (\mathbb{R}^{m-1}\times(0,\infty))$ for any $\Omega\subset\mathbb{R}^m$.

The following computations take place in local coordinates $x \colon \mathcal{U} \to \mathscr{U}$ as discussed in Section \ref{metriccoord}.
We now fix these local coordinates and work in $\mathscr{U}$, regarding $g$ as a Riemannian metric on $\mathscr{U}$. We also fix
a point $x_0 \in \mathscr{U} \cap (\mathbb{R}^{m - 1} \times \{0\})$ and assume that $g_{ij}(x_0) = \delta_{ij}$ and that
\eqref{metricchartbounds} and \eqref{metricchartderivbounds} are satisfied.
 
\begin{lem}\label{enmon}
There exists a constant $C>0$ such that the following holds. Suppose that $v \in W^{1,2}_{\beta}(B^+_{R}(x_0);\mathcal{N})$ satisfies \eqref{innervar} for all
$\psi \in C_0^{\infty}(\mathscr{U};\mathbb{R}^m)$ with $\psi(x',0) \subset\mathbb{R}^{m-1}\times\{0\}$ for every $(x',0) \in \mathscr{U} \cap(\mathbb{R}^{m-1}\times\{0\})$. Let $R > 0$ such that $B_{2R}^+(x_0) \subset \mathscr{U}$.  
  Then for every $0<t<r\leq R$ 
\begin{align*}
\lefteqn{2  \int_{ B^+_{r}(x_0)\backslash B^+_{t}(x_0) } x_m^{\beta} \vert x -x_0\vert^{2-m-\beta}e^{C\vert x - x_0\vert}(1+C\vert x - x_0\vert)^{-1}    \left\vert\frac{(x - x_0)_i}{\vert x-x_0\vert}\partial_iv\right\vert^2 \mathrm{dvol}_{g}} \qquad \nonumber\\
 &  \leq  r^{2-m-\beta}e^{Cr}  \int_{B^+_{r}(x_0)}  x_m^{\beta} \vert Dv\vert^2_g\mathrm{dvol}_{g}  - t^{2-m-\beta}e^{Ct}  \int_{B^+_{t}(x_0)}  x_m^{\beta} \vert Dv\vert^2_g\mathrm{dvol}_{g}.
\end{align*}
In particular the map $\rho \mapsto  \rho^{2-m-\beta}e^{C\rho}  \int_{B^+_{\rho}(x_0)}  x_m^{\beta}\vert Dv\vert^2_g\mathrm{dvol}_{g}$ is non-decreasing in $\rho$ for $0< \rho < R$.
\end{lem}
\begin{proof}
We may assume without loss of generality that $x_0 = 0$.
Let $\eta \in C_0^{\infty}(B_R(0);\mathbb{R})$ and define  $\psi \in C_0^{\infty}(B_R(0);\mathbb{R}^m)$ by $\psi(x) = x\eta(x)$. Then $\psi_m(x',0) = 0$ so $\psi(x',0) \subset  \mathbb{R}^{m-1}\times\{0\} $ for every $(x',0) \in  \mathbb{R}^{m-1}\times\{0\}$. We test \eqref{innervar} with $\psi$ and rearrange to see that 
\begin{align}\label{innervarprelimtestrearrange}
\lefteqn{(\beta + m -2)  \int_{B^+_R(0)}\eta x_m^{\beta} \vert Dv\vert^2_{g}\mathrm{dvol}_{g}} \qquad \nonumber\\
& =  2\int_{B^+_R(0)} x_m^{\beta} g^{ij}\langle \partial_iv   ,\partial_{k}v   \rangle x_k\partial_j\eta \mathrm{dvol}_{g}
- \int_{B^+_R(0)} x_m^{\beta} \vert Dv\vert^2_{g}x_k\partial_k\eta \mathrm{dvol}_{g}  \nonumber\\
& \quad -    \int_{B^+_R(0)} x_k\eta x_m^{\beta}  \left(   g^{ij}   \frac{\partial_k\text{det}(g)}{2\text{det}(g)}   +      \partial_k g^{ij}    \right) \langle \partial_iv   ,\partial_jv  \rangle \mathrm{dvol}_{g}.
\end{align}	

Now let $\chi$ be a smooth cutoff function with $\chi \equiv 1$ in $[1,\infty)$, $\chi \equiv 0$ in $(-\infty,0]$ and such that $0 \leq \chi \leq 1$ and $\vert\chi'\vert \leq 3$. Define $\eta_l(x) = \chi(l(\rho - \vert x \vert))$.   Replacing $\eta$ with $\eta_l$ in \eqref{innervarprelimtestrearrange} and letting  $l \to \infty$ using Lebesgues Dominated Convergence Theorem and Lebesgue's Differentiation Theorem together we see that 
\begin{align}\label{innervarprelimtestrearrangesubincutofflim}
\lefteqn{(\beta + m -2)  \int_{B^+_{\rho}(0)}  x_m^{\beta}\vert Dv\vert^2_{g}\mathrm{dvol}_{g}} \qquad \nonumber\\
& =  -2 \rho^{-1}\int_{\partial^+ B^+_{\rho}(0)} x_m^{\beta} g^{ij}\langle \partial_iv   ,\partial_{k}v   \rangle x_k x_j   \mathrm{d}{\mathcal{S}_g}(\rho) + \rho\int_{\partial^+ B^+_{\rho}(0)} x_m^{\beta} \vert Dv\vert^2_{g}    \mathrm{d}\mathcal{S}_{g}(\rho)   \nonumber\\
& \quad -    \int_{B^+_{\rho}(0)} x_k  x_m^{\beta}  \left(   g^{ij}   \frac{\partial_k\text{det}(g)}{2\text{det}(g)}   +      \partial_k g^{ij}    \right) \langle \partial_iv   ,\partial_jv  \rangle \mathrm{dvol}_g
\end{align}	
for almost every $\rho <R$, where $\mathrm{d}\mathcal{S}_g(\rho) = \sqrt{\det(g)} d\mathcal{H}^{m - 1}$ is the surface measure
with density $\sqrt{\det(g)}$ relative to the $(m - 1)$-dimensional Hausdorff measure. 

The function $\rho \mapsto \int_{B^+_{\rho}(0)}  x_m^{\beta}\vert Dv\vert^2_{g}\mathrm{dvol}_{g} $ is differentiable at Lebesgue-almost every $\rho \in (0,R)$ with derivative $\int_{\partial^+B^+_{\rho}(0)}  x_m^{\beta}\vert Dv\vert^2_{g}\mathrm{d}\mathcal{S}_g(\rho)$. Furthermore, by assumption $g^{ij}(0) = \delta_{ij}$ and thus we have that $\vert g^{ij}(x) x_j - x_i\vert \leq C\rho^2$.
It follows that for almost every $\rho<R$
 \begin{align}\label{innervarprelimtestrearrangesubincutofflimbound}
 & 2 \rho \int_{\partial^+ B^+_{\rho}(0)} x_m^{\beta}  \left\vert\frac{x_i}{\vert x \vert}\partial_iv\right\vert^2 \mathrm{d}\mathcal{S}_g(\rho)   \nonumber\\
 & =    (2-m-\beta)  \int_{B^+_{\rho}(0)}  x_m^{\beta}\vert Dv\vert^2_{g}\mathrm{dvol}_{g} + \rho\int_{\partial^+ B^+_{\rho}(0)} x_m^{\beta}\vert Dv\vert^2_{g} \mathrm{d}\mathcal{S}_g(\rho)   \nonumber\\
 & \quad -    \int_{B^+_{\rho}(0)} x_k  x_m^{\beta}  \left(   g^{ij}   \frac{\partial_k\det(g)}{2\det(g)}   +      \partial_k g^{ij}    \right) \langle \partial_iv   ,\partial_jv  \rangle \mathrm{dvol}_{g} \nonumber\\
 & \quad -     2 \rho^{-1}\int_{\partial^+ B^+_{\rho}(0)} x_m^{\beta}  \langle \partial_iv   ,\partial_{k}v   \rangle x_k (g^{ij}x_j - x_i) \mathrm{d}\mathcal{S}_g(\rho) 
 \nonumber\\
 & \leq    ( 2-m-\beta +C\rho) \int_{B^+_{\rho}(0)}  x_m^{\beta}\vert Dv\vert^2_{g}\mathrm{dvol}_{g} + (1+C\rho)\rho\int_{\partial^+ B^+_{\rho}(0)} x_m^{\beta} \vert Dv\vert^2_{g}  \mathrm{d}\mathcal{S}_g({\rho}) \nonumber\\
 & \leq  (1+C\rho)\left(( 2-m-\beta  +C\rho) \int_{B^+_{\rho}(0)}  x_m^{\beta}\vert Dv\vert^2_{g}\mathrm{dvol}_{g} +  \rho\int_{\partial^+ B^+_{\rho}(0)} x_m^{\beta} \vert Dv\vert^2_{g}  \mathrm{d}\mathcal{S}_g({\rho})\right)\nonumber\\
 & =  (1+C\rho) e^{-C\rho}\rho^{m+\beta-1}\partial_{\rho}\left( e^{C\rho}\rho^{2-m-\beta}  \int_{B^+_{\rho}(0)}  x_m^{\beta}\vert Dv\vert^2_{g}\mathrm{dvol}_{g} \right)
 \end{align}
 for almost every $\rho<R$, where $C$ is a suitable constant. 
 Now we multiply by $  (1+C\rho)^{-1} \rho^{1-m-\beta}e^{C\rho} $ and integrate between $0<t<r<R$ to conclude the proof. 
\end{proof}

In addition to the monotonicity formula from Lemma \ref{enmon}, we can make use of the monotonicity formula
derived by Gro{\ss}e-Brauckmann \cite{MR1184789} for balls in the interior of $\mathcal{M}$. As a consequence,
if we have control of the energy in some ball, then we have control of the energy in smaller balls contained therein, too.
We now briefly switch our focus to geodesic balls $B_r^g(p)$ in $\mathcal{M}$ with respect to $g$ in order to emphasise that the
corresponding inequalities are independent of the choice of local coordinates. To simplify the notation, we
introduce the function
\[
\omega(p, r) = \int_{B_r^g(p)} d^\beta \mathrm{dvol}_g, \quad p \in \mathcal{M}, \ r > 0.
\]

\begin{cor}\label{monotonicity_inequality}
Suppose that $\mathcal{K} \subset \mathcal{M}$ is compact. Then there exist two constants $C, R > 0$ such that
the following holds true. Let $v \in W_{\mathrm{loc}}^{1,2}(\mathcal{M};\mathcal{N})$ be a weakly stationary harmonic
map with respect to free boundary data. Then for any $p \in \mathcal{K}$ and for $0 < t \le r \le R$
\[
\frac{t^2}{\omega(p, t)} \int_{B_t^g(p)} d^\beta |Dv|_g^2 \mathrm{dvol}_g \le \frac{Cr^2}{\omega(p, r)} \int_{B_r^g(p)} d^\beta |Dv|_g^2 \mathrm{dvol}_g.
\]
\end{cor}

\begin{proof}
We  observe that $\omega(p, r)$ is of order $r^m$ if $p$ is far away from the boundary and of order $r^{m + \beta}$
for $p \in \partial \mathcal{M}$. Thus in both cases, the scaling factors in front of the integrals are consistent with
Gro{\ss}e-Brauckmann's monotonicity formula and Lemma \ref{enmon}, respectively. More generally, we can estimate
\[
\frac{1}{c} \omega(p, r) \le r^m (\max\{r, d(p)\})^\beta \le c\omega(p, r)
\]
throughout $\mathcal{K}$ for some constant $c$.

If $4t > r$, then the inequality from the lemma is immediate. Therefore, we assume that $4t \le r$. Choose $R$ so small that for every $p \in \mathcal{K} \cap \partial\mathcal{M}$, the ball $B_{4R}(p)$ is contained in
a coordinate chart as in Section \ref{metriccoord}.

We first consider $p \in \mathcal{K} \cap \partial\mathcal{M}$ (unless the set is empty). We
choose local coordinates $x \colon \mathcal{U} \to \mathscr{U}$ with the properties of Section \ref{metriccoord},
including \eqref{coordmetricbounds}, and set $x_0 = x(p) \in \mathscr{U}$. Then Lemma \ref{enmon} applies. Using the notation from the lemma, we conclude that
\[
t^{2 - m -\beta} \int_{B_t^+(x_0)} x_m^\beta |Dv|_g^2 \mathrm{dvol}_g \le C r^{2 - m -\beta} \int_{B_r^+(x_0)} x_m^\beta |Dv|_g^2 \mathrm{dvol}_g
\]
for some constant $C$ independent of $p$. Here we work on Euclidean balls depending on the choice of local
coordinates. However, according to the discussion in Section \ref{metriccoord}, in particular by \eqref{comparisonofballs},
if we choose $R$ sufficiently small, then
\[
B_{s/2}^g(p) \subset B_s^+(x_0) \subset B_{2s}^g(p)
\]
for any $s \le R$. Then
\[
\begin{split}
t^{2 - m -\beta} \int_{B_t^g(x_0)} x_m^\beta |Dv|_g^2 \mathrm{dvol}_g & \le t^{2 - m -\beta} \int_{B_{2t}^+(x_0)} x_m^\beta |Dv|_g^2 \mathrm{dvol}_g \\
& \le C \left(\frac{r}{2}\right)^{2 - m -\beta} \int_{B_{r/2}^+(x_0)} x_m^\beta |Dv|_g^2 \mathrm{dvol}_g \\
& \le C2^{m + \beta - 2}r^{2 - m -\beta} \int_{B_r^g(x_0)} x_m^\beta |Dv|_g^2 \mathrm{dvol}_g
\end{split}
\]
as required.

If $p \in \mathcal{K}$ satisfies $r/2 \le d(p) \le R$, then we can still work in local coordinates as above, but constructed
about a point $p' \in \partial \mathcal{M}$. We set $x_0 = x(p)$ again, which is now in the interior of $\mathscr{U}$.
We consider the Riemannian metric $\tilde{h}(x) = r^{-\alpha} h(x_0 + rx)$. Then $\tilde{h}$ and all of its derivatives
are uniformly bounded in a Euclidean ball $B_{r_0}(0)$ for a fixed radius $r_0 > 0$ that is independent of $p$.
The map $x \mapsto v(x_0 + rx)$ is weakly stationary harmonic with respect to the metric $\tilde{h}$, and we can
use Gro{\ss}e-Brauckmann's monotonicity formula in $B_{r/4}^g(p)$ to derive the estimate
\[
\begin{split}
\frac{t^2}{\omega(p, t)} \int_{B_t^g(p)} d^\beta |Dv|_g^2 \mathrm{dvol}_g & \le \frac{C r^2}{\omega(p, r/4)} \int_{B_{r/4}^g(p)} d^\beta |Dv|_g^2 \mathrm{dvol}_g \\
& \le \frac{C r^2}{\omega(p, r)} \int_{B_r^g(p)} d^\beta |Dv|_g^2 \mathrm{dvol}_g.
\end{split}
\]

For points with $d(p) \ge R$, the behaviour of the metric near the boundary becomes completely irrelevant. In this
case, we can use the usual arguments for harmonic maps.

We hence consider the case when $d(p) < r/2$. Here we choose $p' \in \partial{M}$
with $\mathrm{dist}(p, p') = d(p)$ and we use the inequalities that we already know. Set $s = d(p)/2$. If
$3t \le s$, we conclude that
\[
\begin{split}
\frac{t^2}{\omega(p, t)} \int_{B_t^g(p)} d^\beta |Dv|_g^2 \mathrm{dvol}_g & \le \frac{C s^2}{\omega(p, s)} \int_{B_s^g(p)} d^\beta |Dv|_g^2 \mathrm{dvol}_g \\
& \le \frac{C (3 s)^2}{\omega(p, 3s)} \int_{B_{3s}^g(p')} d^\beta |Dv|_g^2 \mathrm{dvol}_g \\
& \le  \frac{C (r/2)^2}{\omega(p, r/2)} \int_{B_{r/2}^g(p')} d^\beta |Dv|_g^2 \mathrm{dvol}_g \\
& \le \frac{C r^2}{\omega(p, r)} \int_{B_r^g(p)} d^\beta |Dv|_g^2 \mathrm{dvol}_g.
\end{split}
\]

If $3t > s$, then an inequality similar to the first line holds trivially and we can still use the rest of the argument.
\end{proof}

When we work with local coordinates, the following version of Corollary \ref{monotonicity_inequality} is
convenient. Here we fix the same local coordinates as at the beginning of this section again, and we regard
$g$ as a Riemannian metric on the domain $\mathscr{U}$.

\begin{cor} \label{monotonicity_inequality_Euclidean}
There exists a constant $C>0$ with the following property. Let $R > 0$ such that $B_{2R}^+(x_0) \subset \mathscr{U}$. 
Suppose that $v \in W^{1,2}_{\beta}(B^+_{R}(x_0);\mathcal{N})$ is a weakly stationary harmonic map with respect
to free boundary data. Let $y \in B_{R/2}(x_0)$ and $r > 0$ such that $B_r(y) \subset B_R(x_0)$.
If $y_m \le r$, then
\[
r^{2 - m - \beta} \int_{B_r(y) \cap \mathscr{U}} x_m^\beta |Dv|_g^2 \mathrm{dvol}_g \le CR^{2 - m - \beta} \int_{B_R^+(x_0)} x_m^\beta |Dv|_g^2 \mathrm{dvol}_g.
\]
If $y_m \ge r$, then
\[
r^{2 - m} y_m^{-\beta} \int_{B_r(y) \cap \mathscr{U}} x_m^\beta |Dv|_g^2 \mathrm{dvol}_g \le CR^{2 - m - \beta} \int_{B_R^+(x_0)} x_m^\beta |Dv|_g^2 \mathrm{dvol}_g.
\]
\end{cor}

\begin{proof}
This follows from Corollary \ref{monotonicity_inequality} and the fact that for the given local coordinates,
Euclidean balls are comparable to geodesic balls with inclusions similar to \eqref{comparisonofballs}.
\end{proof}

Using these monotonicity inequalities we can deduce an estimate in the space of bounded mean oscillation for a stationary weakly harmonic map in terms of the re-scaled energy. We recall here that for open $\Omega \subset \mathbb{R}^m$, the mean oscillation of a map $v:\Omega \to \mathbb{R}^n$ in $\Omega$ is 
\begin{equation*}
[v]_{BMO(\Omega)} = \sup_{B\subset \Omega} \frac{1}{\vert B\vert}\int_{B}\vert v - \overline{v}_B\vert \mathrm{d}x
\end{equation*}
where $B\subset \Omega$ is any Euclidean ball, $\vert B\vert = \int_{B}\mathrm{d}x$ and $\overline{v}_{B} = \frac{1}{\vert B\vert}\int_{B}v\mathrm{d}x$. Then  $BMO(\Omega) := \{v\in L^1_{loc}(\Omega):[v]_{BMO(\Omega)}<\infty \}$.
We also introduce the notation $\vert\Omega\vert_{\beta} = \int_{\Omega}\vert x_m\vert^{\beta}\mathrm{d}x$ and $\overline{v}_{\Omega,\beta} = \frac{1}{\vert \Omega\vert_{\beta}}\int_{\Omega}\vert x_m\vert^{\beta}v\mathrm{d}x$.

In the following lemma, in addition to the map $v$ on $\mathscr{U}$, we also consider the even reflection $\tilde{v}$ on
$\mathscr{V} = \mathscr{U} \cup \mathscr{U}_-$ as defined  in Section \ref{metriccoord}.

\begin{lem}\label{statharmbmo}
There exists a constant $C>0$ with the following property. Let $R > 0$ such that $B_{2R}(x_0) \subset \mathscr{V}$.  
Suppose that $v \in W^{1,2}_{\beta}(B^+_{R}(x_0);\mathcal{N})$ is a weakly stationary harmonic map with respect to
free boundary data. Then
\begin{equation}\label{statlocalbmoboundreflec}
 	[\tilde{v}]_{BMO(B_\frac{R}{2}(x_0))}\leq C\left( R^{2-m-\beta}\int_{B^+_R(x_0)}x_m^{\beta}\vert Dv\vert^2_{g} \mathrm{dvol}_{g}\right)^{\frac{1}{2}}.
 	\end{equation}
\end{lem}
\begin{proof}
We let  $B_r(y) \subset B_{R/2}(x_0)$ with centre $y = (y',y_m)\in\mathbb{R}^{m-1}\times\mathbb{R}$.
We may assume without loss of generality that $y_m \geq 0$.

If $y_m \ge r$, then the Poincar\'e inequality, combined with Corollary \ref{monotonicity_inequality_Euclidean},
implies that
\[
\frac{1}{\vert B_r(y)\vert} \int_{B_r(y)}\vert v - \overline{v}_{B_r(y)}\vert \mathrm{d}x \le C \left( R^{2-m-\beta}\int_{B^+_R(x_0)}x_m^{\beta}\vert Dv\vert^2_{g} \mathrm{dvol}_{g}\right)^{\frac{1}{2}}.
\]
  For the case $y_m < r$, we need to replace the usual Poincar\'e inequality
with something else. Here we note instead that 	
\begin{align*}
\frac{1}{\vert B_r(y)\vert} \int_{B_r(y)}\vert \tilde{v} - \overline{\tilde{v}}_{B_r(y)}\vert \mathrm{d}x  & \leq C r^{1-m}    \int_{B_r(y_0)}\vert D\tilde{v}\vert \mathrm{d}x\nonumber\\
 	& \leq Cr^{1-m}     \left(\int_{B_r(y_0)}|x_m|^{-\beta} \mathrm{d}x\right)^{\frac{1}{2}} \left(\int_{B_r(y_0)}|x_m|^{\beta}\vert D\tilde{v}\vert^2 \mathrm{d}x\right)^{\frac{1}{2}}\nonumber\\
 	& \le   C   \left(r^{2-m-\beta}   \int_{B_r(y_0)}x_m^{\beta}\vert D\tilde{v}\vert^2 \mathrm{d}x\right)^{\frac{1}{2}}.
\end{align*}
We apply Corollary \ref{monotonicity_inequality_Euclidean} to complete the proof.
\end{proof}

\section{Control of the $L^2$-norm}\label{l2}

Here, and for most of the rest of the paper, we work in local coordinates as in Section \ref{metriccoord}
again. We establish control of the quantity
\[
\int_{B_r^+(x_0)} x_m^\beta |v - c|^2 \, dx,
\]
for certain constants $c$ and suitable radii $r$, for a solution $v$ to the Euler-Lagrange equations for $E_{\beta}$ in terms of
the energy in a larger ball, provided the latter energy is assumed small. The control we obtain will allow us to control an $L^2$ term which appears in the Caccioppoli type inequality in Section \ref{caccdecay} provided the energy is sufficiently small.  We establish this control for general target manifolds using a blow-up argument.

First we prove an auxiliary estimate on the energy decay for solutions to the associated linear equations with homogeneous Neumann-type boundary conditions. 
\begin{lem}\label{lindecay}
Let $\mathscr{U} \subset \mathbb{R}^{m-1}\times[0,\infty)$ be relatively open. Let $G = (G^{ij}) \colon \mathscr{U} \to \mathbb{R}^{m \times m}$
be a continuous matrix valued function such that $G(x)$ is symmetric and positive definite in $\mathscr{U}$.
Suppose further that $G^{im}(x', 0) = G^{mi}(x', 0) = 0$ for all $(x, 0) \in \mathscr{U} \cap (\mathbb{R}^{m - 1} \times \{0\})$.
Then for any compact set $\mathscr{K} \subset \mathscr{U}$ there exist constants $C, C', \gamma > 0$ with the following properties.
Let $y \in (\mathbb{R}^{m-1}\times\{0\})$ and $\rho > 0$ such that $B_\rho^+(y) \subset \mathscr{K}$.
Suppose $v \in W^{1,2}_{\beta}(\mathscr{U};\mathbb{R}^n)$ satisfies 
\begin{equation}\label{linweaksol}
0 = \int_{\mathscr{U}}  x_m^{\beta}  G^{ij}\langle \partial_iv, \partial_j\psi  \rangle\mathrm{d}x
\end{equation}
for every $\psi \in C_0^{\infty}(\mathscr{U};\mathbb{R}^n)$. Then there exists $\lambda \in \mathbb{R}^n$ such that
\begin{equation}\label{lincaccineq}
\int_{B^+_{\rho/2}(y)} x_m^{\beta} |Dv|^2 \mathrm{d}x\leq \frac{C}{\rho^2}\int_{B^+_{\rho}(y)}x_m^{\beta}\vert v - \lambda\vert^2\mathrm{d}x \le C' \rho^{m + \beta - 2 + 2\gamma}.
\end{equation}
\end{lem}

\begin{proof}
Let $\mathscr{V} = \{(x', x_m) \in \mathbb{R}^m \colon (x', |x_m|) \in \mathscr{U}\}$. We extend $v$ to $\mathscr{V}$
by even reflection. We also extend $G$ to $\mathscr{V}$, using even reflection of $G^{ij}$ for $1 \le i, j \le m - 1$ or
$i = j = m$ and odd reflection otherwise. Then $v$ is a weak solution of the equation
\begin{equation} \label{reflected_equation}
\partial_i (|x_m|^\beta G^{ij} \partial_j v) = 0 \quad \text{in $\mathscr{V}$}.
\end{equation}

Let $\eta \in C_0^{\infty}(B_{\rho}(y))$ with $\eta \equiv 1$ in $B_{\rho/2}(y)$, $0 \leq \eta \leq 1$, $\eta \equiv 0$ in $\mathbb{R}^m \backslash B_{\rho}(y)$ and $\vert D\eta \vert \leq \frac{2}{\rho}$.  It is well-known that $\vert x_m\vert^{\beta}$ is a Muckenhoupt weight of class $A_2$, see  \eqref{Apcond} for the definition,  which yields density of $C^{\infty}(B_{\rho}(y);\mathbb{R}^n)\cap W^{1,2}_{\beta}(B_{\rho}(y);\mathbb{R}^n)$ in $W^{1,2}_{\beta}(B_{\rho}(y);\mathbb{R}^n)$, for instance as a result of  Corollary 2.1.6   in \cite{MR1774162} Chapter 2.
It follows that $\eta^2 (v-\lambda)$ is an admissible test function for \eqref{linweaksol}.   Hence, using Young's inequality, we see that 
\begin{align}\label{linweaksoltestyoungapp}
\int_{B^+_{\rho}(y)}  \eta^2 x_m^{\beta} G^{ij}\langle \partial_iv, \partial_jv \rangle \mathrm{d}x
	&  = -2\int_{B^+_{\rho}(y)}   x_m^{\beta}  G^{ij}\langle \partial_iv, \eta\partial_j\eta (v-\lambda)  \rangle \mathrm{d}x\nonumber\\
	& \leq \frac{1}{2} \int_{B^+_{\rho}(y)}  \eta^2 x_m^{\beta} G^{ij}\langle \partial_iv, \partial_jv \rangle \mathrm{d}x\nonumber\\
	& \quad    + C \rho^{-2}\int_{B^+_{\rho}(y)}   x_m^{\beta} |G|  \vert v-\lambda\vert^2 \mathrm{d}x.
	\end{align}
In the compact set $\mathscr{K}$ (and in its reflection), we have a uniform upper bound for $G$ and a uniform, positive lower bound
for the lowest eigenvalue of $G$. Hence the first inequality in the statement follows.

For the second inequality, it now suffices to choose $\lambda = v(y)$ and
verify that $v$ is H\"older continuous in $\mathscr{K}$.
To this end, we apply the regularity theory of \cite{MR643158}, in particular Theorem 2.3.12 from \cite{MR643158},
to equation \eqref{reflected_equation} in $\mathscr{V}$.
\end{proof}

Now we prove an estimate for the $L^2$-norm of weak solutions of the Euler-Lagrange equations for $E_{\beta}$. 
Recall the notation
\[
\overline{v}_{\Omega,\beta} = \frac{\int_{\Omega}v\vert x_m\vert^{\beta}\mathrm{d}x}{\int_{\Omega} \vert x_m\vert^{\beta}\mathrm{d}x} 
\]
for $\Omega \subset \mathbb{R}^m$. The following result applies to fixed local coordinates as in Section \ref{metriccoord},
as usual, and therefore we assume that $\mathscr{U} \subset \mathbb{R}^{m - 1} \times [0, \infty)$ is
relatively open and $g$ is a Riemannian
metric on $\mathscr{U}$.

\begin{lem}\label{bdrycontimppoincscal}
Let $\mathscr{K} \subset \mathscr{U}$ be a compact set.
For every $\delta,C^{*} >0$  there exist constants $\varepsilon_0$, $\theta_0 \in (0,\frac{1}{8}]$ and $R_0 > 0$ such that the following holds.
Suppose $v \in W^{1,2}_{\beta}(\mathscr{U};\mathbb{R}^n)$ satisfies
	\begin{align}\label{wELeqnbound}
	\left\vert \int_{\mathscr{U}} x_m^{\beta} g^{ij} \langle  \partial_i\psi   ,\partial_jv\rangle \mathrm{dvol}_{g}\right\vert \leq C^{*}  \int_{\mathscr{U}}  \vert \psi\vert x_m^{\beta} \vert Dv\vert^2_{g}  \mathrm{dvol}_{g} 
	\end{align}
	for every $\psi \in C_0^{\infty}(\mathscr{U},\mathbb{R}^n)$. Suppose $B^+_R(x_0) \subset \mathscr{K}$ is a
	half-ball. If $R \leq R_0$ and
	\begin{equation*}
	R^{2-m-\beta}\int_{B^+_R(x_0)} x_m^{\beta}\vert Dv\vert^2_{g}\mathrm{dvol}_{g} \leq \varepsilon_0
	\end{equation*}
	then
	\begin{equation*}
	(\theta_0 R)^{-(m+\beta)}\int_{B^+_{\theta_0 R}(x_0)}  x_m^{\beta} \left\vert   v - \overline{v}_{B^+_{\theta_0 R}(x_0),\beta}\right\vert^2\mathrm{d}x  \leq \delta R^{2-m-\beta}\int_{B^+_R(x_0)}  x_m^{\beta} \vert Dv\vert^2_{g}\mathrm{dvol}_{g} .
	\end{equation*}
\end{lem}

\begin{proof}
	The proof of the lemma is based on a blow-up procedure, analogous to that of the proof of Lemma 3.5 in \cite{moser2005partial}, for example. 
   Suppose, for a contradiction, that there exist $\delta,C^{*} > 0$ such that the claim is false. Then for any $\theta_0 \in (0,\frac{1}{8}]$ there is a sequence of maps $(v_k)_{k\in\mathbb{N}}$, with $v_k \in  {W}^{1,2}_{\beta}(\mathscr{U};\mathbb{R}^n)$ for every $k$, such that
	\begin{align}\label{wELeqnboundcontradictionseq}
	\left\vert \int_{\mathscr{U}} x_m^{\beta} g^{ij} \langle  \partial_i\psi   ,\partial_jv_k\rangle \mathrm{dvol}_{g}\right\vert \leq C^{*}  \int_{\mathscr{U}}  \vert \psi\vert x_m^{\beta} \vert Dv_k\vert^2_{g} \mathrm{dvol}_{g} 
\end{align}
for every $\psi \in C_0^{\infty}(\mathscr{U};\mathbb{R}^n)$ and a sequence of half-balls $B^+_{R_k}(y_k) \subset \mathscr{K}$ such that $R_k \to 0^+$ and
	\begin{align*}
\varepsilon_k^2 := R_k^{2-m-\beta}\int_{B^+_{R_k}(y_k)} x_m^{\beta}\vert Dv_k\vert^{2}_{g}\mathrm{dvol}_{g}  \to 0\ \text{as}\ k \to \infty
	\end{align*}
	but 
	\begin{align}\label{contpreseqboundbelow}
	(\theta_0 R_k)^{-(m+\beta)}\int_{B^+_{\theta_0 {R_k}}(y_k)} x_m^{\beta} \vert  v_k - \overline{(v_k)}_{B^+_{\theta_0 {R_k}}(y_k),\beta} \vert^2 \mathrm{d}x &   >  \delta \varepsilon_k^2.
	\end{align}
	
	Consider the normalised sequence $(w_k)_{k \in\mathbb{N}}$ with $w_k \in W^{1,2}_{\beta}(B^+_1(0);\mathbb{R}^n)$ defined by
	$
	w_k = {\varepsilon_k^{-1}}( v_k(R_kx+y_k)  - \overline{(v_k)}_{B^+_{\theta R_k}(y_k),\beta})
	$.
	Then 
	$
 \partial_i w_k(x)  = R_k\varepsilon_k^{-1}\partial_i v_k(R_kx+y_k)
	$
	and thus
	\begin{equation}\label{cont1rescenandcontzerorescav}
	\int_{B^+_1(0)}  x_m^{\beta} \vert D w_k\vert^{2}\mathrm{d}x   \leq C \quad \text{and} \quad \overline{(w_k)}_{B^+_{\theta_0  }(0),\beta}  =0,
	\end{equation}
	where $C$ is a constant depending on the Riemannian metric and on $\mathscr{K}$.
	Furthermore, we deduce from \eqref{contpreseqboundbelow} that
	\begin{equation}\label{contseqboundbelow}
	\theta_0^{-(m+\beta)}  \int_{B^+_{\theta_0 }(0  )}  x_m^{\beta} \left\vert w_k\right\vert^2\mathrm{d}x > \delta.
	\end{equation}
	Recall that $\vert x_m\vert^{\beta}$ is an $A_2$ weight as defined in \eqref{Apcond}. 
	Hence we may apply  \eqref{cont1rescenandcontzerorescav} and the Poincar\'e  inequality for $A_2$ weights to deduce that $(w_k)_{k \in \mathbb{N}}$ is bounded in  $W^{1,2}_{\beta}(B^+_1(0);\mathbb{R}^n)$. Using  Lemma 2.5 of \cite{roberts2018regularity}, we find a subsequence $(w_{k_l})_{l \in \mathbb{N}}$ which converges weakly in $W^{1,2}_{\beta}(B^+_1(0);\mathbb{R}^n)$ and strongly in $L^2_{\beta}(B^+_1(0);\mathbb{R}^n)$ to a map $w \in W^{1,2}_{\beta}(B^+_1(0);\mathbb{R}^n)$.
	
	Let $g_{k} = g(R_kx+y_k)$. Then in view of \eqref{wELeqnboundcontradictionseq} and \eqref{cont1rescenandcontzerorescav}  we calculate
	\begin{align*}
	\left\vert\int_{B^+_{1}( 0)} x_m^{\beta} g^{ij}_{k}\langle \partial_iw_k  ,\partial_j\psi  \rangle\sqrt{\text{det}(g_k)}\mathrm{d}x\right\vert  
	\leq  C^{*}\vert\vert \psi\vert\vert_{L^{\infty}(B^+_1(0);\mathbb{R}^n)}\varepsilon_k 
	\end{align*}
	for every    $\psi \in C^{\infty}_0(B_1(0);\mathbb{R}^n)$.
Recall that $\mathscr{K}$ is compact.
Hence, taking a convergent  subsequence of   $(y_{k_l})_{l \in \mathbb{N}}$ and re-indexing, we may assume that
$y_{k_l} \to \tilde{y}$, where $\tilde{y}_m = 0$.
Now we observe that the continuity of $g$ yields  $g_{k_l}(x) \to F := g(\tilde{y})$ uniformly. 
 
Using   the weak convergence of $(w_{k_l})_{l \in \mathbb{N}}$ to $w$, it follows that
\begin{align*}
	\left\vert \int_{B^+_{1}( 0)} x_m^{\beta} F^{ij}\langle \partial_i w ,\partial_j\psi  \rangle\sqrt{\text{det}(F)} \mathrm{d}x\right\vert & = \lim_{l\to \infty}\left\vert \int_{B^+_{1}( 0)} x_m^{\beta} g_{k_l}^{ij}\langle \partial_iw_{k_l}  ,\partial_j\psi  \rangle\sqrt{\text{det}(g_{k_l})} \mathrm{d}x \right\vert\nonumber\\
	 &  \leq C^{*}\vert\vert \psi\vert\vert_{L^{\infty}(B^+_1(0);\mathbb{R}^n)}\lim_{l \to \infty} \varepsilon_{k_l} =0
	\end{align*}
for every $\psi \in C_0^{\infty}(B_1(0);\mathbb{R}^n)$.  Hence $w$ satisfies \eqref{linweaksol} from Lemma \ref{lindecay} in $B^+_1(0)$
for the constant function $G = F^{-1} \sqrt{\det(F)}$. We apply the lemma to see that there exist $C>0$ and $\gamma \in (0,1)$
such that
\begin{equation}\label{linendecayapphalf}
	\theta_0^{2-m-\beta}\int_{B^+_{\theta_0}(0)} x_m^{\beta}|Dw|^2 \mathrm{d}x\leq C\theta_0^{2\gamma}.
	\end{equation}
	We also conclude, using strong convergence of $(w_{k_l})_{l \in \mathbb{N}}$ to $w$
	in $L^2_{\beta}(B^+_1(0);\mathbb{R}^n)$ to take limits in \eqref{contseqboundbelow}, that
		\begin{equation}\label{contlimseqboundbelow}
	\theta_0^{-( m+\beta)}\int_{B^+_{\theta_0 }(0)} x_{m}^{\beta}\left\vert w \right\vert^2\mathrm{d}x \geq \delta .
	\end{equation}
Now, since $\overline{w}_{B^+_{\theta_0 }(0),\beta}   =0$ , the Poincar\'e  inequality for $A_2$ weights  yields 
	\begin{align}\label{contlimsolpoinc}
	\theta_0^{-(m+\beta)} \int_{B^+_{\theta_0}(0)} x_{m}^{\beta}\vert w\vert^2\mathrm{d}x  \leq C\theta_0^{2-m-\beta}\int_{B^+_{\theta_0}(0)}  x_{m}^{\beta}  \vert   D w  \vert^2  \mathrm{d}x.
	\end{align}
Inequalities \eqref{linendecayapphalf}, \eqref{contlimseqboundbelow}, and \eqref{contlimsolpoinc} finally
imply that $\delta \le C\theta_0^{2\gamma}$.
If $\theta_0$ is chosen sufficiently small, then this is a contradiction.
\end{proof}

 \section{Maps into Spheres}\label{spheres}
 
 All of our results in previous sections apply to (free boundary) harmonic maps into a general smooth compact manifold $\mathcal{N}$. In this section and henceforth we specialise to the case where $\mathcal{N}  = \mathbb{S}^{n-1} = \{y \in \mathbb{R}^n:\vert y\vert = 1\}$ is the round $n$-sphere embedded isometrically in $\mathbb{R}^{n}$. In this case
a weakly harmonic map $v \in W^{1,2}_{\beta}(\mathscr{U};\mathbb{S}^{n-1})$ with respect to free boundary data satisfies
 \begin{equation}\label{wELeqnsphere}
 0 =  \int_{\mathscr{U}} x_m^{\beta} (g^{ij}\langle  \partial_i\psi   ,\partial_jv\rangle - |Dv|_g^2 \langle  v,\psi\rangle)\mathrm{dvol}_{g}   
 \end{equation}
 for every $\psi \in C_0^{\infty}(\mathscr{U};\mathbb{R}^n)$.

 In order to establish partial regularity,  as in the case for stationary harmonic maps into spheres, we take advantage of the fact that the Euler-Lagrange equations can be rewritten in a form reminiscent of a conservation law. We write \eqref{wELeqnsphere} essentially in `div-curl' form but we also have boundary conditions coming from the free boundary.

 \begin{lem}\label{conslawElequiv}
 	A map $v  \in {W}^{1,2}_{\beta}(\mathscr{U};\mathbb{S}^{n-1})$ satisfies \eqref{wELeqnsphere} for every $\psi \in C_0^{\infty}(\mathscr{U};\mathbb{R}^n)$ if and only if the vector fields $X_{ab} \in L^2_{-\beta}(\mathscr{U};\mathbb{R}^m)$ defined in components by 
 	\begin{equation}\label{divfreeveccomp}
 	X^i_{ab} =  x_m^{\beta}\sqrt{\det(g)}g^{ij} (v^a \partial_jv^b   -  v^b \partial_jv^a)
 	\end{equation}
 	satisfy
 	\begin{equation}\label{divequationwithneumann}
 	\int_{\mathscr{U}} \partial_i \psi X_{ab}^i  \mathrm{d}x=0
 	\end{equation}
 	for $a,b = 1,\ldots,n$ and every $\psi \in C_0^{\infty}(\mathscr{U};\mathbb{R})$.
 \end{lem}
 \begin{rem}
 	Lemma \ref{conslawElequiv} asserts that $v \in W^{1,2}_{\beta}(\mathscr{U};\mathbb{S}^{n-1})$ is weakly harmonic with respect to free
 	boundary data if and only if the vector fields $X_{ab}\in L^2_{-\beta}(\mathscr{U};\mathbb{R}^n)$ satisfy the equation
 	$\text{div}(X_{ab}) = 0$ in $\mathrm{int}(\mathscr{U})$ with boundary condition $X_{ab} = 0$ on $\mathscr{U}\cap(\mathbb{R}^{m-1}\times\{0\})$ weakly. 
 \end{rem}
 \begin{proof}
 	We follow the proof of Lemma 3.5 in \cite{moser2005partial}. Suppose $v \in W^{1,2}_{\beta}(\mathscr{U};\mathbb{S}^{n-1})$ satisfies \eqref{wELeqnsphere} for every $\psi \in C_0^{\infty}(\mathscr{U};\mathbb{R}^n)$. Then, by
Lemma \ref{test_functions},    $v$ satisfies \eqref{wELeqnsphere} for every $\psi \in W^{1,2}_{\beta}(\mathscr{U};\mathbb{R}^n) \cap L^{\infty}(\mathscr{U};\mathbb{R}^n) $ with support compactly contained in $\mathscr{U}$. Then, for $\psi \in C_0^{\infty}(\mathscr{U};\mathbb{R})$, we conclude that the map $\psi v^a \in W^{1,2}_{\beta}(\mathscr{U};\mathbb{R}) \cap L^{\infty}(\mathscr{U};\mathbb{R}) $ is an admissible test function and it follows that
 	\begin{align}\label{weakdivfreecomp}
 \int_{\mathscr{U}} \partial_i \psi X_{ab}^i  \mathrm{d}x
 	& =  \int_{\mathscr{U}} x_m^{\beta} g^{ij} ( \partial_i (\psi v^a )\partial_jv^b   -  \partial_i(\psi v^b) \partial_jv^a)  \mathrm{dvol}_{g} \nonumber\\
 	& =  \int_{\mathscr{U}}x_m^{\beta}\vert Dv\vert^2_{g}\psi(v^av^b-v^av^b)    \mathrm{dvol}_{g}  = 0.
 	\end{align}
 	
 	Conversely, suppose that \eqref{divequationwithneumann}  is satisfied for  $a,b = 1,\ldots,n$ and every $\psi \in C_0^{\infty}(\mathscr{U};\mathbb{R})$.
 	Now let $\psi \in C_0^{\infty}(\mathscr{U};\mathbb{R}^n)$ with components $\psi^b\in C_0^{\infty}(\mathscr{U};\mathbb{R})$. Recall that for each $i = 1,\ldots,m$ we have $v^a\partial_iv^a = 0$ almost everywhere since $\vert v\vert^2 = 1$ almost everywhere. Furthermore, observe that by approximation $v^a\psi^b$ is an admissible test function for  \eqref{divequationwithneumann} for every $a,b=1,\ldots,n$. Hence, summing over $a,b = 1,\ldots,n$, we find  
 	\begin{align}\label{conservimpELsphere}
 	0 & = \int_{\mathscr{U}} \partial_i (v^a\psi^b) (x_m^{\beta}\sqrt{\text{det}(g)}g^{ij} (v^a \partial_jv^b   -  v^b \partial_jv^a))  \mathrm{d}x\nonumber\\
 	& =   \int_{\mathscr{U}} x_m^{\beta}\sqrt{\det(g)}g^{ij} \left(\psi^b\partial_i v^a (v^a \partial_jv^b   -  v^b \partial_jv^a) + v^a\partial_i \psi^b (v^a \partial_jv^b   -  v^b \partial_jv^a)\right)  \mathrm{d}x\nonumber\\
 	& =   \int_{\mathscr{U}}   x_m^{\beta} g^{ij}  (\langle\partial_i\psi,\partial_jv\rangle - \langle\partial_i v ,\partial_jv\rangle  \langle v,\psi\rangle )  \mathrm{dvol}_{g} 
 	\end{align}
 	as required.
 \end{proof}

 \section{Hardy Space Estimate With Weights}\label{hardy}
 
 In order to establish partial regularity for sphere-valued stationary harmonic maps we follow the approach in \cite{moser2005partial} Section 3
 (based on ideas by H\'elein \cite{helein1990regularite} and Evans \cite{MR1143435}). In particular, we combine the estimate given by Lemma \ref{bdrycontimppoincscal} for the mean squared oscillation with a Caccioppoli-Type inequality to show that if the re-scaled energy is sufficiently small, then it decays faster than implied by the monotonicity formula in Lemma \ref{enmon}. In order to derive this inequality, we will take advantage of the equation
derived in Lemma \ref{conservimpELsphere} satisfied by $v$. It implies, roughly speaking, that the product $Dv\cdot X$ lies,
after extending by a reflection with respect to $\mathbb{R}^{m - 1} \times \{0\}$, in a Hardy space.

 We recall the definition of the Hardy space $\mathcal{H}^1(\mathbb{R}^m)$. For more details, see
 \cite{MR1232192}, for example.
 Let $\mathcal{F}$ be the collection of all $\psi \in C_0^{\infty}(B_1(0);\mathbb{R})$ with $\sup_{B_1(0)}\vert D\psi\vert \leq 1$.  Define $\psi_s(x) = s^{-m}\psi(s^{-1}x)$ for $\psi \in \mathcal{F}$ and $s >0$. For a locally integrable distribution $v$ on $\mathbb{R}^m$ let $\psi_s * v = \int_{\mathbb{R}^m}\psi_s(x-y)v(y)\mathrm{d}y  $ and define  the maximal function 
 \begin{equation*}
 {M}v(x) = \sup_{\psi \in \mathcal{F}}\sup_{s>0}\vert (\psi_s * v)(x)\vert
 \end{equation*}
 for $x \in \mathbb{R}^m$.
 The Hardy Space $\mathcal{H}^1(\mathbb{R}^m)$ then consists of locally integrable distributions $v$ with norm $\vert\vert v\vert \vert_{\mathcal{H}^1(\mathbb{R}^m)} = \vert\vert  {M}v\vert \vert_{L^1(\mathbb{R}^m)}<\infty$.

 In order to prove our Hardy space estimate, we will need to utilise the theory of $A_p$ weights, which we define and discuss presently.  Suppose $\mu$ and $\nu$ are measures and let $\mathrm{d}\nu = w\mathrm{d}\mu$ for a non-negative locally integrable function $w$. Then $w$ is in $A_p(\mathrm{d}\mu)$ if 
 \begin{align}\label{Apcond}
 \frac{1}{\mu(B)}\int_{B}w\mathrm{d}\mu \left(\frac{1}{\mu(B)}\int_{B}w^{-\frac{1}{p-1}}\mathrm{d}\mu\right)^{p-1} \leq c
 \end{align}
 for every ball $B \subset\mathbb{R}^m$, see \cite{MR1011673} Chapter I.
 
 For the $A_p$ class of weights, many counterparts to estimates in unweighted Sobolev-Spaces remain true. We recall the estimates we need here.  Let $\mathrm{d}x$ denote the $m$-dimensional Lebesgue measure. We use the notation $W^{1,p}(\mathscr{U};w\mathrm{d}x)$ to denote weighted Sobolev spaces comprised of classes of $\mathrm{d}x$ measurable functions with weak first order derivatives such that the function and all its first order derivatives are $p$-integrable with respect to the measure $w\mathrm{d}x$. Recall that smooth functions contained in $W^{1,p}(\mathscr{U};w\mathrm{d}x)$ are dense provided $w$ is an $A_p$ weight, see Corollary 2.1.6 of \cite{MR1774162}. We further have the following Sobolev-Poincar\'e inequality. 
 \begin{lem}[\cite{MR643158} Theorem 1.5]\label{weightedsobolevpoinc}
 	Let  $w \in A_p(\mathrm{d}x)$ where $1<p<\infty$.   Then there exist $C,\delta>0$ such that for $1\leq \kappa \leq \frac{m}{m-1}+\delta$ and any $v$ which is Lipschitz continuous on the closure of $B_r(x_0)\subset\mathbb{R}^m$  we have 
 	\begin{equation}\label{weightedsobolevpoincestdense}
 	\left(\frac{1}{\int_{B_r(x_0)}w\mathrm{d}x}\int_{B_r(x_0)}\vert v - \overline{v}\vert^{\kappa p}w\mathrm{d}x\right)^{\frac{1}{\kappa p}} \leq Cr \left(\frac{1}{\int_{B_r(x_0)}w\mathrm{d}x}\int_{B_r(x_0)}\vert Dv\vert^{ p}w\mathrm{d}x\right)^{\frac{1}{ p}} 
 	\end{equation}
 	where we may choose either $\overline{v} = \frac{1}{\int_{B_r(x_0)}w\mathrm{d}x}\int_{B_r(x_0)}vw\mathrm{d}x$ or $\overline{v} = \frac{1}{\int_{B_r(x_0)} \mathrm{d}x}\int_{B_r(x_0)}v\mathrm{d}x$. Furthermore, \eqref{weightedsobolevpoincestdense} extends by density  to hold for   $v \in W^{1,p}(B_r(x_0);w\mathrm{d}x)$ by Corollary 2.1.6 of \cite{MR1774162}. 
 \end{lem}
 We observe that the bound $\frac{m}{m-1}+\delta$ may not allow the unweighted Sobolev exponent $\kappa = \frac{m}{m-p}$, however the bound given in the Lemma is sufficient for our purpose. A more precise bound is given in Chapter 15 of \cite{MR1207810}.

 We define the Hardy-Littlewood maximal operator corresponding to a measure $\mathrm{d}\mu$ by $M_{\mu}(f)(x) = \sup_{B \ni x}\frac{1}{\mu(B)}\int_B \vert f\vert\mathrm{d}\mu$. For $A_p$ weights, the following version of the Hardy-Littlewood Maximal Theorem holds. 
 \begin{lem}[\cite{MR1011673} Theorem 9]\label{ApHardylittlewood}
 	Let $p > 1$ and suppose $\mathrm{d}\nu = w\mathrm{d}\mu$ where $w \in A_p(\mathrm{d}\mu)$. Let $M_{\mu}$ denote the Hardy Maximal function associated to $\mu$. Then $\vert \vert M_{\mu}(f)\vert\vert_{L^p(\mathrm{d}\nu)} \leq C_p \vert \vert f\vert\vert_{L^p(\mathrm{d}\nu)}$ for every $f \in L^p(\mathrm{d}\nu)$.
 \end{lem}

 Now we are in a position to prove our weighted Hardy estimates.  We first establish an estimate for the product of a gradient and a quantity which is divergence free in $\mathbb{R}^m$. We use $\cdot $ to denote the inner product in $\mathbb{R}^m$ for brevity.
 \begin{lem}\label{weightedHardy}
 	Let $\beta \in (-1,1)$ and $m \geq 3$. Then there exists $C>0$ such that if $v \in W^{1,2}_{\beta}(\mathbb{R}^m)$ and  $X \in L^2_{-\beta}(\mathbb{R}^m;\mathbb{R}^m)$   satisfies $\mathrm{div}(X ) = 0$ weakly in $\mathbb{R}^m$ then $D v \cdot X  \in \mathcal{H}^1(\mathbb{R}^m)$ and 
 	\begin{equation*}
 	\vert \vert Dv \cdot X  \vert\vert_{\mathcal{H}^1(\mathbb{R}^m)}\leq C \vert \vert Dv\vert\vert_{L^2_{\beta}(\mathbb{R}^m)}\vert \vert X \vert\vert_{L^2_{-\beta}(\mathbb{R}^m)}. 	 
 	\end{equation*}
 \end{lem}
\begin{rem}
 We will require certain weight functions to satisfy corresponding $A_p$ conditions for the proof of Lemma \ref{weightedHardy} but postpone the proof of these conditions to a subsequent Lemma, see Lemma \ref{Apformaximal}. 
\end{rem} 
 
 \begin{proof}[Proof of Lemma \ref{weightedHardy}]
The following arguments are inspired by the work of Coifman-Lions-Meyer-Semmes \cite{coifman1990compacite}.

 	Let $\psi_s(x) = s^{-{m}}\psi(\frac{x}{s})$ for $s > 0$ and where $\psi \in C_0^{\infty}(B_1(0);\mathbb{R})$ with $\vert D\psi\vert \leq 1$. Then we have 
 	\begin{align*}
 	\psi_s*(D v \cdot X ) & = s^{-m}\int_{B_s(x)}\psi\left(\frac{x-y}{s}\right)D v \cdot X  \mathrm{d}y\nonumber\\
 	& =  -s^{-m-1}  \int_{B_s(x)}D\psi\left(\frac{x-y}{s}\right)  (v-\overline{v}) \cdot X \mathrm{d}y.
 	\end{align*} 
  Now fix   $q \in (2, \frac{2m}{m-1})$. Then $2 > \frac{q}{q-1}>1$ and $2 > q\frac{m-1}{m} > 1 $.
Furthermore, we check that  
\[
(\beta  + 1)\frac{ (m-1)}{2m}  -\frac{1}{q}   <  \min\left\{\frac{1-\frac{2}{ q} +   \beta}{2},1 - \frac{m+q}{mq}\right\}.
\]
Hence we may choose $\gamma \neq 0$ such that  
 	\begin{equation}\label{Hardygammachoicerec}
 	(\beta  + 1)\frac{ (m-1)}{2m}  -\frac{1}{q}   <  \gamma        < \min\left\{\frac{1-\frac{2}{ q} +   \beta}{2},1 - \frac{m+q}{mq}\right\}.
 	\end{equation}  
 	 Let $\overline{v}$ be the average of $v$ on $B_s(x)$ with respect to $\vert y_m\vert^{\gamma q}\mathrm{d}y$.
Furthermore, we write
\[
|\Omega|_\sigma = \int_\Omega |y_m|^\sigma \, dy 
\]
for measurable sets $\Omega \subset \mathbb{R}^m$.
We have 
 	\begin{equation}\label{Hardysameweightsobolevpoincestweightpre}
-s^{-m - 1} \int_{B_s(x)}D\psi\left(\frac{x-y}{s}\right)  (v-\overline{v}) \cdot X \mathrm{d}y
 	\leq  \frac{C  }{s\vert B_s(x)\vert }\int_{B_s(x)}  \vert v-\overline{v} \vert \vert X\vert\mathrm{d}y.
 	\end{equation}
 	Then, using H\"older's inequality for the Lebesgue measure we see that,
 	\begin{align}\label{Hardysameweightsobolevpoincestweight}
 	& \frac{C  }{s\vert B_s(x)\vert }\int_{B_s(x)}  \vert v-\overline{v} \vert \vert X\vert\mathrm{d}y\nonumber\\ 
 	& =\frac{C  }{s\vert B_s(x)\vert }\int_{B_s(x)} \vert y_m\vert^{\gamma}\vert y_m\vert^{-\gamma} \vert v-\overline{v} \vert \vert X\vert\mathrm{d}y\nonumber\\
 	& \leq \frac{C}{s\vert B_s(x)\vert }\left(\int_{B_s(x)}\vert y_m\vert^{q\gamma}\vert v - \overline{v}\vert^q\mathrm{d}y\right)^{\frac{1}{q}}\left(\int_{B_s(x)}\vert y_m\vert^{-\gamma\frac{q}{q-1}}\vert X\vert^\frac{q}{q-1}\mathrm{d}y\right)^{\frac{q-1}{q}}\nonumber\\
 	& = \frac{C\vert B_s(x)\vert_{\gamma q}^{\frac{1}{q}}\vert B_s(x)\vert_{-\gamma \frac{q}{q-1}}^{\frac{q-1}{q}} }{s\vert B_s(x)\vert }\left(\frac{1}{\vert B_s(x)\vert_{\gamma q}}\int_{B_s(x)}\vert y_m\vert^{\gamma q}\vert v - \overline{v}\vert^q\mathrm{d}y\right)^{\frac{1}{q}}\nonumber\\
 	& \quad \cdot\left(\frac{1}{\vert B_s(x)\vert_{-\gamma \frac{q}{q-1}}}\int_{B_s(x)}\vert y_m\vert^{-\gamma \frac{q}{q-1}}\vert X\vert^\frac{q}{q-1}\mathrm{d}y\right)^{\frac{q-1}{q}}.
 	\end{align}
 	Note that since  $\gamma$ satisfies \eqref{Hardygammachoicerec}, the right hand side of \eqref{Hardysameweightsobolevpoincestweight} is finite.  Now we observe that 
 	\begin{align}\label{weightedballisAq}
 \frac{ \vert B_s(x)\vert_{\gamma q}^{\frac{1}{q}}\vert B_s(x)\vert_{-\gamma \frac{q}{q-1}}^{\frac{q-1}{q}} }{ \vert B_s(x)\vert } 
 	& = 
 	\left(\frac{ \int_{B_s(x)}\vert y_m\vert^{\gamma q}\mathrm{d}y }{ \vert B_s(x)\vert }\left( \frac{   \int_{B_s(x)}\vert y_m\vert^{-\gamma \frac{q}{q-1}}\mathrm{d}y  }{ \vert B_s(x)\vert } \right)^{q-1}\right)^{\frac{1}{q}} 
 	\end{align}
 	which is precisely the quantity in the condition that $\vert y_m\vert^{\gamma q}$ is an $A_q(\mathrm{d}y)$ weight (to the power of $\frac{1}{q}$). Since $\gamma$ satisfies \eqref{Hardygammachoicerec},  we have  $\gamma \in (-\frac{1}{q}, 1 - \frac{1}{q})$ and  Lemma \ref{Apformaximal}  implies that  $\vert y_m\vert^{\gamma q}$ is an $A_q(\mathrm{d}y)$ weight. Hence, 
 	\begin{align}\label{weightedballisAqbound}
 	\frac{ \vert B_s(x)\vert_{\gamma q}^{\frac{1}{q}}\vert B_s(x)\vert_{-\gamma \frac{q}{q-1}}^{\frac{q-1}{q}} }{ \vert B_s(x)\vert } 
 	 \leq C(m,q,\gamma).
 	\end{align}
 	  We  apply Lemma \ref{weightedsobolevpoinc} with $p = q\frac{m-1}{m}$ and $\kappa = \frac{m}{m-1}$ so that $\kappa p = q$, noting that our choice of $\gamma$ and Lemma \ref{Apformaximal} imply   $\vert y_m\vert^{\gamma q}$ is an $A_{q\frac{m-1}{m}}(\mathrm{d}y)$ weight.  This yields
 	  \begin{align}\label{sobolevpoincapp}
 	  \frac{1}{s}\left(\frac{1}{\vert B_s(x)\vert_{\gamma q}}\int_{B_s(x)}\vert y_m\vert^{\gamma q}\vert v - \overline{v}\vert^q\mathrm{d}y\right)^{\frac{1}{q}}\leq C \left(\frac{1}{\vert B_s(x)\vert_{\gamma q}}\int_{B_s(x)}\vert y_m\vert^{\gamma q}\vert Dv\vert^{\frac{q(m-1)}{m}}\mathrm{d}y\right)^{\frac{m}{q(m-1)}}.
 	  \end{align}
 	  We combine \eqref{Hardysameweightsobolevpoincestweightpre}, \eqref{Hardysameweightsobolevpoincestweight}, \eqref{weightedballisAqbound} and \eqref{sobolevpoincapp} to see that 
 	\begin{align}\label{Hardysameweightsobolevpoincestweightappsobpoincboundbymaximal}
 	\lefteqn{-s^{-m-1}\int_{B_s(x)}D\psi\left(\frac{x-y}{s}\right)  (v-\overline{v}) \cdot X\mathrm{d}y} \qquad \nonumber\\
 	& \leq   C  (M_{\gamma q}(\vert Dv\vert^{\frac{q(m-1)}{m}})(x))^{\frac{m}{q(m-1)}}(M_{-\gamma \frac{q}{q-1}}(\vert X\vert^{\frac{q}{q-1}})(x))^{\frac{q-1}{q}},
 	\end{align}
 	where we use the notation $M_{\gamma q}$ and $M_{-\gamma\frac{q}{q-1}}$ for the  Maximal functions  $M_{\mu_1}$ and $M_{\mu_2}$, as defined prior to Lemma \ref{ApHardylittlewood}, corresponding to the measures $\mathrm{d}\mu_1 = \vert y_m\vert^{\gamma q}\mathrm{d}y$ and $\mathrm{d}\mu_2 = \vert y_m\vert^{-\gamma\frac{q}{q-1}}\mathrm{d}y$  respectively.

 It follows from \eqref{Hardysameweightsobolevpoincestweightappsobpoincboundbymaximal} and H\"older's inequality  that
 	\begin{align}\label{estimateofHardynorm}
 	&\left\vert\left\vert \sup_{\psi \in \mathcal{F}}\sup_{s>0} \psi_s*(\nabla v \cdot X)\right\vert\right\vert_{L^1(\mathbb{R}^m)}\nonumber\\
 	& \leq C\int_{\mathbb{R}^m}    (M_{\gamma q}(\vert Dv\vert^{\frac{q(m-1)}{m}}))^{\frac{m}{q(m-1)}}(M_{-\gamma \frac{q}{q-1}}(\vert X\vert^{\frac{q}{q-1}}))^{\frac{q-1}{q}}\mathrm{d}x\nonumber\\
 	& = C\int_{\mathbb{R}^m}  \vert x_m\vert^{\frac{\beta}{2}}  (M_{\gamma q}(\vert Dv\vert^{\frac{q(m-1)}{m}}))^{\frac{m}{q(m-1)}}\vert x_m\vert^{-\frac{\beta}{2}}(M_{-\gamma \frac{q}{q-1}}(\vert X\vert^{\frac{q}{q-1}}))^{\frac{q-1}{q}}\mathrm{d}x\nonumber\\
 	& \leq C\left(\int_{\mathbb{R}^m}  \vert x_m\vert^{\beta}  (M_{\gamma q}(\vert Dv\vert^{\frac{q(m-1)}{m}}))^{2\frac{m}{q(m-1)}} \mathrm{d}x\right)^{\frac{1}{2}}\left(\int_{\mathbb{R}^m}   \vert x_m\vert^{-\beta}(M_{-\gamma \frac{q}{q-1}}(\vert X\vert^{ \frac{q}{q-1}}))^{2\frac{q-1}{q}}\mathrm{d}x\right)^{\frac{1}{2}}.
 	\end{align}
 	
 	Now we apply Lemma \ref{ApHardylittlewood} to two sets of measures. We write  $\mathrm{d}\mu_1 = \vert x_m\vert^{ \gamma q}\mathrm{d}y$ and  $\mathrm{d}\nu_1 = \vert x_m\vert^{\beta - \gamma q}\mathrm{d}\mu_1 = \vert x_m\vert^{\beta }\mathrm{d}y$. It follows from Lemma \ref{Apformaximal} that $\vert x_m\vert^{\beta - \gamma q}$ is an $A_{2\frac{m}{q(m-1)}}(\mathrm{d}\mu_1)$ weight, hence Lemma \ref{ApHardylittlewood} yields
 	\begin{align}\label{Hardylittlewoodappliedtosobolev}
 	\int_{\mathbb{R}^m}  \vert x_m\vert^{\beta}  (M_{\gamma q}(\vert Dv\vert^{\frac{q(m-1)}{m}}))^{2\frac{m}{q(m-1)}} \mathrm{d}x   \leq  C\int_{\mathbb{R}^m}  \vert x_m\vert^{\beta}  \vert Dv\vert^{2}  \mathrm{d}x.
 	\end{align}
 	Similarly we write $\mathrm{d}\mu_{2} = \vert x_m\vert^{-\gamma \frac{q}{q-1}}\mathrm{d}x$ and $\mathrm{d}\nu_2 = \vert x_m\vert^{ \gamma \frac{q}{q-1}-\beta}\mathrm{d}\mu_2 = \vert x_m\vert^{-\beta}\mathrm{d}x$. It follows from Lemma \ref{Apformaximal} that $\vert x_m\vert^{ \gamma \frac{q}{q-1}-\beta}$ is an   $A_{2\frac{q-1}{q}}(\mathrm{d}\mu_2)$ weight and hence it follows from   Lemma \ref{ApHardylittlewood} that
 	\begin{align}\label{Hardylittlewoodappliedtoconj}
 	\int_{\mathbb{R}^m}   \vert x_m\vert^{-\beta}(M_{-\gamma \frac{q}{q-1}}(\vert X\vert^{ \frac{q}{q-1}}))^{2\frac{q-1}{q}}\mathrm{d}x \leq C\int_{\mathbb{R}^m}   \vert x_m\vert^{-\beta}\vert X\vert^{2}  \mathrm{d}x.
 	\end{align}
 	The conclusion of the Lemma follows from \eqref{estimateofHardynorm}, \eqref{Hardylittlewoodappliedtosobolev} and \eqref{Hardylittlewoodappliedtoconj}.
 \end{proof}

Next we prove a local version of Lemma \ref{weightedHardy}.
\begin{lem}\label{weightedHardyloc}
	Let $\beta \in (-1,1)$, $m \geq 3$, $v \in W^{1,2}_{\beta}(B_r(x_0);\mathbb{R})$ and suppose that $X  \in L^2_{-\beta}(B_r(x_0);\mathbb{R}^m)$  with $\mathrm{div}(X  ) = 0$ weakly in $B_r(x_0)$.  Let $\eta \in C_0^{\infty}(B_r(x_0);\mathbb{R})$ with $ 0 \leq \eta \leq 1$, $\eta = 1$ in $B_{\frac{r}{2}}(x_0)$ and $\vert D\eta\vert \leq \frac{C_0}{r}$. Extend $Dv$ and $ X $ by $0$ to $\mathbb{R}^m$ and let $\lambda = \left(\int_{\mathbb{R}^m}\eta\mathrm{d}x\right)^{-1}\int_{\mathbb{R}^m}\eta Dv\cdot X \mathrm{d}x$. Then   $\eta ( D v \cdot X -\lambda)  \in \mathcal{H}^1(\mathbb{R}^m)$ and, for some constant $C = C(m,\beta,C_0)$ we have
	\begin{equation*}
	\vert \vert\eta ( D v \cdot X -\lambda) \vert\vert_{\mathcal{H}^1(\mathbb{R}^m)}\leq C \vert \vert Dv\vert\vert_{L^2_{\beta}(B_r(x_0);\mathbb{R}^m)}\vert \vert X  \vert\vert_{L^2_{-\beta}(B_r(x_0);\mathbb{R}^m)}. 	 
	\end{equation*}
\end{lem}
\begin{rem}
	The constant $\lambda$ is subtracted since it is a necessary condition that $\int_{\mathbb{R}^m}g\mathrm{d}x = 0$ for any $g \in \mathcal{H}^1(\mathbb{R}^m)$, see e.g. \cite{semmes1994primer}. 
\end{rem}
\begin{proof}
	Let $\psi \in C_0^{\infty}(B_1(0))$ with $\vert D\psi\vert \leq 1$ and define $\psi_s(x) = s^{-{m}}\psi(\frac{x}{s})$ for $s > 0$. 
	We have
	\begin{align*}
	\psi_s*(\eta ( D v \cdot X -\lambda) )  = \psi_s*(\eta D v \cdot X )  - \lambda\psi_s*\eta .
	\end{align*}	
	Let $\overline{v}\in \mathbb{R}$ denote a constant which will later be chosen to be the average $v$ over an appropriate set.  
    Since $X $ is weakly divergence free in $B_r(x_0)$, we see that for any $x \in \mathbb{R}^m$ we have
	\begin{align}\label{Hardyestintbypartcutoff}
	\psi_s*(\eta D v \cdot X  )(x) 
	& = s^{-m}\int_{B_s(x)\cap B_r(x_0)}\psi\left(\frac{x-y}{s}\right)\eta D v \cdot X   \mathrm{d}y\nonumber\\
	& =  s^{-m - 1}  \int_{B_s(x)\cap B_r(x_0)}D\psi\left(\frac{x-y}{s}\right) \eta (v-\overline{v}) \cdot X  \mathrm{d}y \nonumber\\
	& \quad -   s^{-m}   \int_{B_s(x)\cap B_r(x_0)} \psi\left(\frac{x-y}{s}\right)(v-\overline{v})  D\eta  \cdot X  \mathrm{d}y.
	\end{align} 
	For all $x \in \mathbb{R}^m$, $\overline{v}\in \mathbb{R}$ and $s>0$ we calculate 
	\begin{align}\label{Hardysameweightsobolevpoincestweightloc}
	\lefteqn{\left\vert  s^{-m - 1} \int_{B_s(x)\cap B_r(x_0)}D\psi\left(\frac{x-y}{s}\right) \eta (v-\overline{v}) \cdot X  \mathrm{d}y\right\vert} \qquad \nonumber\\
	& \leq  \frac{C  }{s\vert B_s(x)\vert }\int_{B_s(x)\cap B_r(x_0)}  \vert v-\overline{v} \vert \vert X \vert\mathrm{d}y.
	\end{align}	
	and, using that $\vert D\eta\vert \leq C_0r^{-1}$, we also find
	\begin{align}\label{Hardysameweightsobolevpoincestweightloccutoff}
	\lefteqn{ \left\vert s^{-m}   \int_{B_s(x)\cap B_r(x_0)} \psi\left(\frac{x-y}{s}\right)(v-\overline{v})  D\eta  \cdot X  \mathrm{d}y\right\vert } \qquad \nonumber\\
	& \leq  \frac{C}{\vert B_s(x)\vert}  \int_{B_s(x)\cap B_r(x_0)} \vert v-\overline{v}\vert   \vert D\eta\vert  \vert X \vert\mathrm{d}y \nonumber\\
	& \leq  \frac{C}{r\vert B_s(x)\vert}  \int_{B_s(x)\cap B_r(x_0)} \vert v-\overline{v}\vert    \vert X \vert\mathrm{d}y .
	\end{align} 
	We further observe that 
	\begin{equation}\label{Hardysameweightsobolevpoincestweightloclambdabound}
	|\lambda\psi_s*\eta(x)| = \left|\lambda s^{-m}\int_{B_s(x)\cap B_r(x_0)}\psi\left(\frac{x-y}{s}\right)\eta\mathrm{d}y\right| \leq C\vert \lambda \vert s^{-m}\vert B_s(x)\cap B_r(x_0)\vert . 
	\end{equation}
	We combine we combine \eqref{Hardyestintbypartcutoff} - \eqref{Hardysameweightsobolevpoincestweightloclambdabound}  to see that for all $x \in \mathbb{R}^m$, $\overline{v}\in \mathbb{R}$ and $s>0$ we have
	\begin{align}\label{Hardysameweightsobolevpoincestweightlocgeneral}
	\left\vert \psi_s*(\eta ( D v \cdot X -\lambda) )(x) \right\vert & \leq  \frac{C  }{ \vert B_s(x)\vert }(s^{-1}+r^{-1})\int_{B_s(x)\cap  B_r(x_0)}  \vert v-\overline{v} \vert \vert X \vert\mathrm{d}y +  C\vert \lambda \vert.  
	\end{align}
	
	To proceed we consider the cases $x \in B_{2r}(x_0)$ and $x \in \mathbb{R}^{m}\backslash B_{2r}(x_0)$ and recall that $\overline{v} \in \mathbb{R}$ is a constant we may choose.
	Suppose $x \in B_{2r}(x_0)$ and let $s>0$.  For $s \leq r$ we have $r^{-1}\leq s^{-1}$ and in this case we set $\overline{v} = \overline{v}_{B_s(x)}$ in \eqref{Hardysameweightsobolevpoincestweightlocgeneral}. If $s>r$ then  $B_s(x)\cap B_r(x_0) \subset B_r(x_0)\subset B_{3r}(x)$ and we set $\overline{v} = \overline{v}_{B_{3r}(x)}$ and observe that  $\vert B_s(x)\vert^{-1} \leq 3^m \vert B_{3r}(x)\vert^{-1}$. Then we see that for $x \in B_{2r}(x_0)$ we have
	\begin{align}\label{Hardysameweightsobolevpoincestweightlocnearsuppeta}
	\left\vert \psi_s*(\eta ( D v \cdot X -\lambda) )(x) \right\vert & \leq  \frac{C  }{t\vert B_t(x)\vert }\int_{B_t(x)}  \vert v-\overline{v}_{B_t(x)} \vert \vert X \vert\mathrm{d}y +  C\vert \lambda \vert,  
	\end{align} 
	where $t = s$ if $s \leq r$ and $t = 3r$ if $s > r$. 
	Now, choose $q$ as in the proof of Lemma \ref{weightedHardy} and $\gamma$ as in \eqref{Hardygammachoicerec}. Noting that $X $ and $Dv$ are extended by $0$ outside $B_r(x_0)$   we conclude from \eqref{Hardysameweightsobolevpoincestweightlocnearsuppeta}  and \eqref{Hardysameweightsobolevpoincestweight}-\eqref{Hardysameweightsobolevpoincestweightappsobpoincboundbymaximal}  that 
	\begin{align}\label{Hardysameweightsobolevpoincestweightlocnearsuppetbound}
	\left\vert \psi_s*(\eta ( D v \cdot X -\lambda) )(x) \right\vert & \leq  C  (M_{\gamma q}(\vert Dv\vert^{\frac{q(m-1)}{m}})(x))^{\frac{m}{q(m-1)}}(M_{-\gamma \frac{q}{q-1}}(\vert X \vert^{\frac{q}{q-1}})(x))^{\frac{q-1}{q}}\nonumber\\
& \quad + C \vert \lambda \vert  
	\end{align}
	for every $x \in B_{2r}(x_0)$ and $s>0$, where $M_{\gamma q}$ and $M_{-\gamma \frac{q}{q-1}}$ are the maximal functions as described in Lemma \ref{weightedHardy}.

	Now we consider $x \in \mathbb{R}^m\backslash B_{2r}(x_0)$. Then $\vert x - x_0\vert \geq 2r$.  We  observe that  when $\vert x - x_0\vert \geq r +s$
	we have  $\psi_s*(\eta ( D v \cdot X -\lambda) )(x) = 0$. Hence we only need to consider the case when $r<s$ and  $2r \leq \vert x - x_0\vert < r+s$ (if $s \leq r$ this condition is empty).   In this case we observe that $s^{-1} < 2\vert x-x_0\vert^{-1}$. Hence, similarly to the proof of Proposition 1.92 of \cite{semmes1994primer}, we take advantage of the fact that $\int_{\mathbb{R}^m}\eta ( D v \cdot X -\lambda)\mathrm{d}y = 0$, the Mean Value Theorem and the fact that $\vert D\psi\vert \leq 1$ to see that
	\begin{align}\label{Hardysameweightsobolevpoincestweightloccutoff2rleqmodxminusx0lerplussintgis0}
	\vert \psi_s*(\eta ( D v \cdot X -\lambda) )(x)\vert  & = s^{-m}\left\vert \int_{\mathbb{R}^m}\psi(\frac{x-y}{s})\eta( D v \cdot X -\lambda)\mathrm{d}y\right\vert \nonumber\\
	& = s^{-m}\left\vert\int_{B_r(x_0)}(\psi(\frac{x-y}{s}) - \psi(\frac{x-x_0}{s}))\eta( D v \cdot X -\lambda)\mathrm{d}y\right\vert \nonumber\\
	& \leq Cs^{-m}  s^{-1}\int_{B_r(x_0)}\vert x_0 - y\vert \vert  D v \cdot X -\lambda\vert \mathrm{d}y\nonumber\\
	& \leq C  r\vert x - x_0\vert^{-(m+1)}\int_{B_r(x_0)}\vert D v \cdot X \vert +\vert \lambda\vert \mathrm{d}y.
	\end{align}
	We apply H\"older's inequality to see that 
	\begin{align}\label{Hardysameweightsobolevpoincestweightloccutoff2rleqmodxminusx0lerplussintgis0holder}
	\lefteqn{\vert \psi_s*(\eta ( D v \cdot X -\lambda) )(x)\vert} \qquad \nonumber\\
	& \leq  C  r\vert x - x_0\vert^{-(m+1)}( \vert \vert Dv\vert\vert_{L^2_{\beta}(B_r(x_0);\mathbb{R}^m)}\vert \vert X \vert\vert_{L^2_{-\beta}(B_r(x_0);\mathbb{R}^m)} + \vert \lambda \vert r^m),
	\end{align}
	for every $x \in \mathbb{R}^m\backslash B_{2r}(x_0)$ and every $s>0$.
	Now we estimate the $\mathcal{H}^1(\mathbb{R}^m)$ norm of $\eta ( D v \cdot X -\lambda) $. We combine \eqref{Hardysameweightsobolevpoincestweightlocnearsuppetbound} and \eqref{Hardysameweightsobolevpoincestweightloccutoff2rleqmodxminusx0lerplussintgis0holder} to see that 
	\begin{align}\label{Hardyloccombineest}
	&\left\vert\left\vert \sup_{\psi \in \mathcal{F}}\sup_{s>0} \psi_s*(\eta ( D v \cdot X -\lambda) )\right\vert\right\vert_{L^1(\mathbb{R}^m)} \nonumber\\ 
	& =  \int_{B_{2r}(x_0)} \vert \sup_{\psi \in \mathcal{F}} \sup_{s>0}\psi_s*(\eta ( D v \cdot X -\lambda) )\vert  \mathrm{d}x \nonumber\\
	& \quad +  \int_{\mathbb{R}^m\backslash B_{2r}(x_0)} \vert \sup_{\psi \in \mathcal{F}}\sup_{s>0}\psi_s*(\eta ( D v \cdot X -\lambda) )\vert \mathrm{d}x\nonumber\\ 
	& \leq C\int_{B_{2r}(x_0)}   (M_{\gamma q}(\vert Dv\vert^{\frac{q(m-1)}{m}})(x))^{\frac{m}{q(m-1)}}(M_{-\gamma \frac{q}{q-1}}(\vert X \vert^{\frac{q}{q-1}})(x))^{\frac{q-1}{q}}  + \vert \lambda \vert\mathrm{d}x\nonumber\\
	& \quad + Cr \int_{\mathbb{R}^m\backslash B_{2r}(x_0)} \vert x - x_0\vert^{-(m+1)}( \vert \vert Dv\vert\vert_{L^2_{\beta}(B_r(x_0);\mathbb{R}^m)}\vert \vert X \vert\vert_{L^2_{-\beta}(B_r(x_0);\mathbb{R}^m)} + \vert \lambda \vert r^m) \mathrm{d}x\nonumber\\
	& \leq C\int_{\mathbb{R}^m}   (M_{\gamma q}(\vert Dv\vert^{\frac{q(m-1)}{m}})(x))^{\frac{m}{q(m-1)}}(M_{-\gamma \frac{q}{q-1}}(\vert X \vert^{\frac{q}{q-1}})(x))^{\frac{q-1}{q}}  \mathrm{d}x + Cr^m\vert \lambda \vert\nonumber\\
	& \quad + Cr( \vert \vert Dv\vert\vert_{L^2_{\beta}(B_r(x_0);\mathbb{R}^m)}\vert \vert X \vert\vert_{L^2_{-\beta}(B_r(x_0);\mathbb{R}^m)} + \vert \lambda \vert r^m) \int_{\mathbb{R}^m\backslash B_{2r}(x_0)} \vert x - x_0\vert^{-(m+1)}  \mathrm{d}x.
	\end{align}
	Now we observe that 
	\begin{align}\label{Hardylocouterintest}
	\int_{\mathbb{R}^m\backslash B_{2r}(x_0)} \vert x - x_0\vert^{-(m+1)}  \mathrm{d}x = 2^{-1}r^{-1}\int_{\mathbb{R}^m\backslash B_{1}(0)} \vert x  \vert^{-(m+1)} = Cr^{-1} \mathrm{d}x= Cr^{-1}
	\end{align}
	and, using H\"older's inequality, that
	\begin{align}\label{Hardyloclambdaest}
	\vert \lambda\vert \leq Cr^{-m}\int_{B_r(x_0)}\vert x_m\vert^{\frac{\beta}{2}}\vert x_m\vert^{-\frac{\beta}{2}}\vert Dv\vert \vert X \vert\mathrm{d}x 
	& \leq Cr^{-m}\vert \vert Dv\vert\vert_{L^2_{\beta}(B_r(x_0);\mathbb{R}^m)}\vert \vert X \vert\vert_{L^2_{-\beta}(B_r(x_0);\mathbb{R}^m) }.
	\end{align}	
	We use \eqref{estimateofHardynorm} - \eqref{Hardylittlewoodappliedtoconj} from the proof of Lemma \ref{weightedHardy}, together with the fact that $Dv$ and $X $ are extended by $0$ in $\mathbb{R}^m\backslash B_r(x_0)$ to see that   
	\begin{align}\label{weightedHardylocmaximalintest}
	\lefteqn{\int_{\mathbb{R}^m}    (M_{\gamma q}(\vert Dv\vert^{\frac{q(m-1)}{m}}))^{\frac{m}{q(m-1)}}(M_{-\gamma \frac{q}{q-1}}(\vert X \vert^{\frac{q}{q-1}}))^{\frac{q-1}{q}}\mathrm{d}x} \qquad \nonumber\\
	& \leq  C\vert \vert Dv\vert\vert_{L^2_{\beta}(B_r(x_0);\mathbb{R}^m)}\vert \vert X \vert\vert_{L^2_{-\beta}(B_r(x_0);\mathbb{R}^m)}.
	\end{align}
	We combine \eqref{Hardyloccombineest}-\eqref{weightedHardylocmaximalintest} to conclude the proof.	
\end{proof}

 In the following lemma we give conditions for the measures discussed in the proof of Lemma \ref{weightedHardy} to satisfy their respective $A_p$ conditions.
 \begin{lem}\label{Apformaximal}
 	Let $m \geq 3$, $q \in (2,\frac{2m}{m-1})$ and $\beta \in (-1,1)$.  Let $\mathrm{d}y$ be the Lebesgue measure on $\mathbb{R}^m$.
Define $\mathrm{d}\mu_1 = \vert y_m\vert^{ \gamma q}\mathrm{d}y$, and $\mathrm{d}\mu_{2} = \vert y_m\vert^{-\gamma \frac{q}{q-1}}\mathrm{d}y$.
Then
\begin{enumerate}
\item for $\gamma \in (-\frac{1}{q}, 1 -\frac{1}{q})$ the weight $\vert y_{m}\vert^{\gamma q}$ is in $A_q(\mathrm{d}y)$;
\item for $\gamma \in (-\frac{1}{q}, 1 - \frac{m+q}{mq})$ the weight  $\vert y_m\vert^{\gamma q}$ is in $A_{q\frac{m-1}{m}}(\mathrm{d}y)$;
\item for $\gamma  > (\beta + 1) \frac{m - 1}{2m} - \frac{1}{q}$, the weight $|y_m|^{\beta - \gamma q}$ is in $A_{2\frac{m}{q(m-1)}}(\mathrm{d}\mu_1)$; and
\item for $\gamma < \min\{\frac{1 + \beta}{2}, 1\} - \frac{1}{q}$, the weight $|y_m|^{ \gamma \frac{q}{q-1}-\beta}$ is in $A_{2\frac{q-1}{q}}(\mathrm{d}\mu_2)$.
\end{enumerate}
\end{lem}
\begin{proof}

To establish the Lemma, we show \eqref{Apcond} is satisfied in each case.
In order to do this we consider balls with $\text{dist}(B_r(x);\mathbb{R}^{m-1}\times\{0\})  = \vert x_m\vert -r  \geq  r$ and balls with $0 \leq \vert x_m\vert < 2r$ separately. The subsequent calculations are similar for all four statements. We give the details for
the first one and leave the rest to the reader, that is, 
 we establish that $\vert y_{m}\vert^{\gamma q}$ is in $A_q(\mathrm{d}y)$.

 Suppose first that $\vert x_m\vert -r  \geq  r$. Then
\[
\frac{|x_m|}{2} \le |y_m| \le 2|x_m|
\]
in $B_r(x)$. Hence
\begin{align}\label{Apcondgammaqintballcond}
\frac{1}{\vert B_r(x)\vert}\int_{ B_r(x)}\vert y_m\vert^{\gamma q}\mathrm{d}y \left(\frac{1}{\vert  B_r(x)\vert}\int_{ B_r(x)}\vert y_m\vert^{-\gamma \frac{q}{q-1}} \mathrm{d}y\right)^{q-1}
& \le C |x_m|^{\gamma q} \left(|x_m|^{-\gamma \frac{q}{q-1}}\right)^{q - 1} \nonumber\\
& \le C.
\end{align}

Now suppose $0 \leq \vert x_m\vert < 2r$. If we write $x = (x', x_m)$, then $B_r(x) \subset B_{3r}(x', 0)$.
Then provided $\gamma q>-1$ and $-\gamma\frac{q}{q-1}>-1$, we have 
\begin{align}\label{Apcondgammaqbdryball}
\lefteqn{\frac{1}{\vert B_r(x)\vert}\int_{ B_r(x)}\vert y_m\vert^{\gamma q}\mathrm{d}y \left(\frac{1}{\vert  B_r(x)\vert}\int_{ B_r(x)}\vert y_m\vert^{-\gamma \frac{q}{q-1}} \mathrm{d}y\right)^{q-1}} \qquad \nonumber\\
& \le Cr^{-m} \int_{B_{3r}(x', 0)} |y_m|^{\gamma q} \mathrm{d}y \left(r^{-m} \int_{B_{3r}(x', 0)} |y_m|^{-\gamma \frac{q}{q - 1}} \mathrm{d}y\right)^{q - 1} \nonumber\\
& \le C r^{-q} \int_0^{3r} t^{\gamma q} \mathrm{d}t \left(\int_0^{3r} t^{-\gamma \frac{q}{q - 1}} \mathrm{d}t\right)^{q - 1} \le C.
\end{align} 
Hence we have verified \eqref{Apcond}.
\end{proof}

 \section{Caccioppoli-Type Inequality and Energy Decay}\label{caccdecay}
 With the Weighted Hardy estimate in Lemma \ref{weightedHardy} in hand, we may prove a Caccioppoli-Type inequality for stationary harmonic maps free boundary data. This estimate gives control of the re-scaled energy on a half-ball in terms of the mean squared oscillation and re-scaled energy on a half-ball of twice the radius. We will then use this lemma to show the re-scaled energy decays faster than implied by the monotonicity formula stated in Section \ref{mon}; we will show the decay we obtain is fast enough to imply H\"older continuity in Section \ref{epsilonreg}. 
 
 We need the following preliminary lemma with regard extending a $BMO(B_r(x_0))$-function to $BMO(\mathbb{R}^m)$ by multiplying by a cutoff function. The following lemma follows from well-known arguments, we give a proof based on e.g  \cite{MR1143435} Lemma 4.1 in order to elucidate the dependence of the constants on the cutoff function.
 \begin{lem}\label{bmoloc}
 	Suppose $v \in BMO(B_r(x^{*}))$ for some ball $B_r(x^{*})\subset\mathbb{R}^m$. Let $\eta \in C_0^{\infty}(B_\frac{3r}{4}(x_0))$ with $\eta \equiv 1$ in $B_{\frac{r}{2}}(x^{*})$, $0 \leq \eta \leq 1$ and $\vert D \eta\vert \leq C^{*}r^{-1}$. Then $w = \eta(v-\overline{v}_{B_r(x^{*})}) \in BMO(\mathbb{R}^m)$ and $[w]_{BMO(\mathbb{R}^m)}\leq C[v]_{BMO(B_r(x^{*}))}$ for a constant $C = C(m)(1+C^{*})$. 
 \end{lem}
 \begin{proof}
 	We assume with no loss of generality that $\overline{v}_{B_r(x^{*})} = 0$. 
 
Observe that
 	\begin{align}\label{bmocutoffinitialest}
 	\frac{1}{\vert B_s(x)\vert}\int_{B_s(x)}\vert w - \overline{w}_{B_s(x)}\vert\mathrm{d}y \leq \frac{2}{\vert B_s(x)\vert}\int_{B_s(x)\cap B_r(x^{*})}\vert w\vert\mathrm{d}y.\nonumber\\
 	\end{align}
    If $s \geq \frac{r}{16\sqrt{m}}$ then it follows from \eqref{bmocutoffinitialest} that
 	\begin{align}\label{bmocutofflargesbound}
 	\frac{1}{\vert B_s(x)\vert}\int_{B_s(x)}\vert w - \overline{w}_{B_s(x)}\vert\mathrm{d}y & \leq \frac{2  }{\vert B_s(x)\vert}\int_{B_s(x)\cap B_r(x^{*})}\vert w\vert\mathrm{d}y\nonumber\\
    &\leq \frac{C}{\vert B_r(x^{*})\vert}\int_{B_r(x^{*})}\vert v \vert\mathrm{d}y \nonumber\\
 	& \leq C[v]_{BMO(B_r(x^{*}))}.  
 	\end{align}

 	To proceed, we consider the cases $x \in B_{\frac{13}{16}r}(x^{*})$ and $x \in \mathbb{R}^m\backslash B_{\frac{13}{16}r}(x^{*})$.  	
 	First suppose $x \in \mathbb{R}^m\backslash B_{\frac{13}{16}r}(x^{*})$. If $s \geq \frac{r}{16\sqrt{m}}$ then \eqref{bmocutofflargesbound} holds and if $s < \frac{r}{16\sqrt{m}}$ then $B_s(x)\cap B_{\frac{3r}{4}}(x^{*}) = \emptyset$ and it follows from \eqref{bmocutoffinitialest} and the fact that $\eta \in C_0^{\infty}(B_{\frac{3r}{4}}(x^{*});\mathbb{R})$ that 
 	\begin{align}\label{bmocutoffnointersectbound}
 	\frac{1}{\vert B_s(x)\vert}\int_{B_s(x)}\vert w - \overline{w}_{B_s(x)}\vert\mathrm{d}y  = 0. 
 	\end{align}

 	Now suppose $x \in B_{\frac{13}{16}r}(x^{*})$. If $s \geq \frac{r}{16\sqrt{m}}$ then \eqref{bmocutofflargesbound} still holds.  
 	Hence we consider $0 \leq s <\frac{r}{16\sqrt{m}}$ and note that in this case $B_s(x)\subset B_{\frac{7}{8}r}(x^{*})$. Observe that 
 	\begin{align*}
 	w - \overline{w}_{B_s(x)} = \eta v -  \overline{\eta v}_{B_s(x)} = \eta(v - \overline{v}_{B_s(x)}) + \vert B_s(x)\vert^{-1}\int_{B_s(x)}v(y)(\eta(z)-\eta(y))\mathrm{d}y.
 	\end{align*}
    Consequently, using the fact that $\vert D\eta\vert \leq C^{*}r^{-1}$, we see that
 	\begin{align}\label{johnnirenpre}
 	&\frac{1}{\vert B_s(x)\vert}\int_{B_s(x)}\vert w - \overline{w}_{B_s(x)}\vert\mathrm{d}z \nonumber\\
 		& \leq [v]_{BMO(B_r(x^{*}))}   +\frac{1}{\vert B_s(x)\vert}\int_{B_s(x)}\left\vert \vert B_s(x)\vert^{-1}\int_{B_s(x)}v(y)(\eta(z)-\eta(y))\mathrm{d}y\right\vert\mathrm{d}z\nonumber\\
 	& \leq [v]_{BMO(B_r(x^{*}))} +\frac{2C^{*}s}{r\vert B_s(x)\vert }\int_{B_s(x)} \vert v\vert \mathrm{d}y.
 	\end{align}
Furthermore, we can use H\"older's inequality to estimate
\begin{equation} \label{L^m}
|B_s(x)|^{-1} \int_{B_s(x)} |v| \mathrm{d}y \le  \left(|B_s(x)|^{-1}\int_{B_{7r/8}(x^*)} |v|^m \mathrm{d}y\right)^{\frac{1}{m}}.
\end{equation}
It follows from the John-Nirenberg inequality, see \cite{giaquinta2013introduction} Corollary 6.22, that
\[
\left(\int_{B_{7r/8}(x^*)} |v|^m \mathrm{d}y\right)^{\frac{1}{m}} \le Cr[v]_{BMO(B_r(x^*))}.
\]
Together with \eqref{johnnirenpre} and \eqref{L^m}, this yields the desired estimate.
\end{proof}

Now we prove our Caccioppoli-type estimate. 
 \begin{lem}\label{cacctypeineq}
 There exists a constant $C>0$ such that the following holds.	Suppose $v \in W^{1,2}_{\beta}(\mathscr{U};\mathbb{S}^{n-1})$ is weakly stationary harmonic with respect to
free boundary data. Let $\tilde{v} \in W^{1,2}_{\beta}(\mathscr{V};\mathbb{S}^{n-1})$ denote the even reflection of $v$ in $\mathbb{R}^{m-1}\times\{0\}$. Then for $B^+_r(x_0)\subset\mathscr{U}$ with $(x_0)_m = 0$, $r>0$ and $B^+_{2r}(x_0)\subset\mathscr{U}$  and any $\delta_0>0$  we have
 	\begin{align}\label{cacctypeineqest}
 	\left(\frac{r}{2}\right)^{2-m-\beta}\int_{B^+_\frac{r}{2}(x_0)} x_m^{\beta} \vert Dv\vert^2\mathrm{d}x&  \leq (C [\tilde{v}]_{BMO(B_r(x_0))} + \delta_0)r^{2-m-\beta}\int_{B^+  _r(x_0)} x_m^{\beta} \vert Dv\vert^2\mathrm{d}x \nonumber\\
 	& \quad +  C\delta_0^{-1}r^{-(m+\beta)}\int_{B^+_r(x_0)} x_m^{\beta} \vert   v - \overline{v}_{B^+_r(x_0),\beta}\vert^2\mathrm{d}x.
 	\end{align}
 \end{lem}
 
 \begin{proof}
 The proof is analogous to the case for stationary harmonic maps, see for instance Lemma 3.8 of \cite{moser2005partial}. 	Let $\eta \in C_0^{\infty}(B_{\frac{3}{4}r}(x_0);\mathbb{R})$ be a smooth cutoff function with $\eta \equiv 1$ in $ B_{\frac{r}{2}}(x_0)$, $0 \leq \eta \leq 1$ and $\vert D\eta\vert \leq \frac{C}{r}$. Henceforth we abbreviate the notation $\overline{\tilde{v}}_{B_r(x_0),\beta}$ to $\overline{\tilde{v}}$ and $\overline{v}_{B^+_r(x_0),\beta}$ to $\overline{v}$. Let $\tilde{g}$ be the reflection of $g$.
It follows from \eqref{metricchartbounds} and the fact that 
 	$\vert \tilde{v}\vert^2 = 1$ almost everywhere that 
 	\begin{align}\label{energydensitywedgebound}
  \tilde{g}^{ij}\langle \partial_i\tilde{v},\partial_j\tilde{v}\rangle   = \frac{1}{2 }\tilde{g}^{ij}(\tilde{v}^a\partial_i\tilde{v}^b -  \tilde{v}^b\partial_i\tilde{v}^a)(\tilde{v}^a\partial_j\tilde{v}^b -  \tilde{v}^b\partial_j\tilde{v}^a).
 	\end{align}
 	Now let $w = \eta(\tilde{v} - \overline{\tilde{v}}) \in W^{1,2}_{\beta,0}(B_r(x_0);\mathbb{R}^n)$.   We calculate 
 	\begin{align}\label{cutoffvecfieldcomp}
 	\partial_i( \tilde{v}^aw^b  - \tilde{v}^bw^a ) & =  \partial_i\eta(\tilde{v}^a(\tilde{v}-\overline{\tilde{v}})^b - (\tilde{v}-\overline{\tilde{v}})^a\tilde{v}^b)\nonumber\\
 	& \quad + \eta (\tilde{v}^a\partial_i\tilde{v}^b - \tilde{v}^b\partial_i\tilde{v}^a) \nonumber\\
 	& \quad +   \partial_i\tilde{v}^aw^b - \partial_i\tilde{v}^bw^a 
 	\end{align}
  for $a,b = 1,\ldots,n$.  We observe that since $\tilde{v},w \in L^{\infty}(\mathscr{V};\mathbb{R}^n)$ it follows that $ \tilde{v}^aw^b - \tilde{v}^bw^a  \in W^{1,2}_{\beta}(B_r(x_0);\mathbb{R})$. Moreover, since the support of $\eta$ is a compact subset of $B_r(x_0)$, we see that $\tilde{v}^aw^b - \tilde{v}^bw^a  \in W^{1,2}_{\beta,0}(B_r(x_0);\mathbb{R})$. We observe that the preceding discussion for $w$ also holds true for $W = \eta w$, including \eqref{cutoffvecfieldcomp} which holds replacing $\eta$ with $\eta^2$, $\partial_i\eta$ with $2\eta\partial_i\eta$ and $w$ with $W$.
  
   Lemma \ref{conslawElequiv} implies that the vector fields $\tilde{X}_{ab} \in L^2_{-\beta}(\mathscr{V};\mathbb{R}^n)$ given in components by $\tilde{X}_{ab}^i = \vert x_m\vert^{\beta}\sqrt{\text{det}(g(x',\vert x_m\vert))}\tilde{g}^{ij} (\tilde{v}^a \partial_j\tilde{v}^b   -  \tilde{v}^b \partial_j\tilde{v}^a)$ are weakly divergence free in $\mathscr{V}$. We use this fact, together with the fact that $\tilde{v}^aW^b - \tilde{v}^bW^a  \in W^{1,2}_{\beta,0}(B_r(x_0);\mathbb{R})$ and \eqref{energydensitywedgebound} and \eqref{cutoffvecfieldcomp}, to see that 
 	\begin{align}\label{caccinitialenest}
 	&\int_{B_\frac{r}{2}(x_0)} \vert x_m\vert^{\beta} \tilde{g}^{ij}\langle \partial_i\tilde{v},\partial_j\tilde{v}\rangle\sqrt{\text{det}(g(x',\vert x_m\vert))}\mathrm{d}x\nonumber\\
 	& \leq \frac{1}{2} \int_{B_r(x_0)}\eta^2   \vert x_m\vert^{\beta}   \tilde{g}^{ij}(\tilde{v}^a\partial_i\tilde{v}^b -  \tilde{v}^b\partial_i\tilde{v}^a)(\tilde{v}^a\partial_j\tilde{v}^b -  \tilde{v}^b\partial_j\tilde{v}^a)\sqrt{\text{det}(g(x',\vert x_m\vert))}\mathrm{d}x  \nonumber\\
 	& =  \frac{1}{2} \int_{B_r(x_0)}\eta^2     (\tilde{v}^a\partial_i\tilde{v}^b -  \tilde{v}^b\partial_i\tilde{v}^a)  \tilde{X}_{ab}^i\mathrm{d}x  \nonumber\\
 	& = -      \int_{B_r(x_0)}   \partial_i\eta(\tilde{v}^a(\tilde{v}-\overline{\tilde{v}})^b - (\tilde{v}-\overline{\tilde{v}})^a\tilde{v}^b)\eta  \tilde{X}_{ab}^i\mathrm{d}x\nonumber\\
 	& \quad -   \frac{1}{2} \int_{B_r(x_0)}  (\partial_i\tilde{v}^aw^b - \partial_i\tilde{v}^bw^a)\eta \tilde{X}_{ab}^i\mathrm{d}x.
 	\end{align}
We apply Young's inequality to see that 
 	\begin{align}\label{caccyoungterm}
 	&  -   \int_{B_r(x_0)}      \partial_i\eta(\tilde{v}^a(\tilde{v}-\overline{\tilde{v}})^b - (\tilde{v}-\overline{\tilde{v}})^a\tilde{v}^b)\eta \tilde{X}_{ab}^i\mathrm{d}x \nonumber\\
 	& \leq C \int_{B_r(x_0)}  \vert x_m\vert^{\beta}   \vert D\eta\vert \vert\tilde{v}-\overline{\tilde{v}}\vert \vert D\tilde{v}\vert\mathrm{d}x\nonumber\\
 	& \leq \delta_0\int_{B^+_r(x_0)}     x_m^{\beta}  \vert Dv\vert^2  \mathrm{d}x +  \frac{C}{\delta_0 r^2}\int_{B^+_r(x_0)}   x_m^{\beta}    \vert  v-\overline{v} \vert^2 \mathrm{d}x,
 	\end{align}
 	where we have also used the fact that 
 	\begin{equation*}
    \int_{B_r(x_0)}   \vert x_m\vert^{\beta}    \vert  \tilde{v}-\overline{\tilde{v}} \vert^2 \mathrm{d}x \leq C\int_{B^+_r(x_0)}   x_m^{\beta}    \vert  v-\overline{v} \vert^2 \mathrm{d}x.
 	\end{equation*} 
   It follows from Lemma \ref{bmoloc} that   $w^a \in BMO(\mathbb{R}^m)$ and $[w^{a}]_{BMO(\mathbb{R}^m)} \leq C_0 [\tilde{v}^{a}]_{BMO(B_r(x_0))}$ for $a = 1,\ldots,m$.  We now apply Lemma \ref{weightedHardyloc} which implies that for each $a,b = 1,\ldots,n$ we have $\eta( D\tilde{v}^a \cdot   \tilde{X}_{ab} - \lambda_{ab}) \in\mathcal{H}^1(\mathbb{R}^m)$ where we have extended $\tilde{X}$ and $D\tilde{v}$ by $0$ to $\mathbb{R}^m$ and  $\lambda_{ab} = \left(\int_{\mathbb{R}^m}\eta\mathrm{d}x\right)^{-1}\int_{\mathbb{R}^m}\eta Dv^a\cdot \tilde{X}_{ab}\mathrm{d}x$. We further note that  $\eta( D\tilde{v}^a \cdot   \tilde{X}_{ab} - \lambda_{ab}) \in L^1(\mathbb{R}^m)$ and $w^a \in L^\infty(\mathbb{R}^m)$ for $a,b = 1,\ldots,m$. We may therefore   apply the duality of $\mathcal{H}^1(\mathbb{R}^m)$ and $BMO(\mathbb{R}^m)$, see \cite{moser2005partial} Lemma 2.3, together with Theorem \ref{weightedHardyloc} and the fact that  $\tilde{X}_{ab}^{i} = -\tilde{X}_{ba}^{i}$, to see that 
 	\begin{align}\label{caccHardyapp}
 	&-   \int_{B_r(x_0)}   (\partial_i\tilde{v}^aw^b - \partial_i\tilde{v}^bw^a)\eta  \tilde{X}^{i}_{ab}\mathrm{d}x \nonumber\\
 	& = -  2\int_{B_r(x_0)} w^bD\tilde{v}^a \cdot \eta  \tilde{X}_{ab}\mathrm{d}x   \nonumber\\
 	& = -  2\int_{B_r(x_0)} w^b\eta( D\tilde{v}^a \cdot   \tilde{X}_{ab} - \lambda_{ab})\mathrm{d}x  -    2\int_{B_r(x_0)} w^b \eta \lambda_{ab} \mathrm{d}x \nonumber\\
 	& \leq  2[w^b]_{BMO(\mathbb{R}^m)} \vert\vert \eta( D\tilde{v}^a \cdot   \tilde{X}_{ab} - \lambda_{ab})\vert\vert_{\mathcal{H}^1(\mathbb{R}^m)} +2 \vert \lambda_{ab}\vert\int_{B_r(x_0)}\vert \tilde{v}^b - \overline{\tilde{v}^b}\vert\mathrm{d}x\nonumber\\
 	& \leq C [w^b]_{BMO(\mathbb{R}^m)}\vert\vert D\tilde{v}^a\vert\vert_{L^2_{\beta}(B_r(x_0);\mathbb{R}^{m})}\vert\vert \tilde{X}_{ab}\vert\vert_{L^2_{-\beta}(B_r(x_0);\mathbb{R}^m)}\nonumber\\
 	& \quad + C \vert\vert D\tilde{v}^a\vert\vert_{L^2_{\beta}(B_r(x_0);\mathbb{R}^{m})}\vert\vert \tilde{X}_{ab}\vert\vert_{L^2_{-\beta}(B_r(x_0);\mathbb{R}^m)}\frac{1}{\vert B_r(x_0)\vert}\int_{B_r(x_0)}\vert \tilde{v}^b - \overline{\tilde{v}^b}\vert\mathrm{d}x\nonumber\\ 
 	& \leq C [\tilde{v}^b]_{BMO(B_r(x_0))}\vert\vert D\tilde{v}^a\vert\vert_{L^2_{\beta}(B_r(x_0);\mathbb{R}^{m})}\vert\vert \tilde{X}_{ab}\vert\vert_{L^2_{-\beta}(B_r(x_0);\mathbb{R}^m)}\nonumber\\
 	& \leq  C [\tilde{v}]_{BMO(B_r(x_0)))}\int_{B^+_r(x_0)} x_m^{\beta}\vert Dv\vert^2\mathrm{d}x.
 	\end{align}
 Together \eqref{caccinitialenest}, \eqref{caccyoungterm} and \eqref{caccHardyapp} yield  \eqref{cacctypeineqest} as required.
 \end{proof}
 
 We now combine the Caccioppoli-type inequality in Lemma \ref{cacctypeineq} with the control of the mean squared oscillation given by Lemma \ref{bdrycontimppoincscal} to see that the re-scaled energy of free boundary stationary harmonic maps decays faster than implied by the energy monotonicity formula.
 
  \begin{lem}\label{impenscalbdry}
 	 There exist constants $R_1>0$, $\varepsilon_1 >0$ and $\theta_1 \in (0,\frac{1}{4})$ such that the following holds.   Suppose $v \in  {W}_{\beta}^{1,2}(\mathscr{U};\mathbb{S}^{n-1})$ is stationary harmonic with respect to free boundary data and consider a half ball $B^+_{R}(x_0)$ with $R \leq R_1$ and $B^+_{2R}(x_0)\subset \mathscr{U}$. If
 	\begin{equation*}
 	R^{2-m-\beta}\int_{B^+_R(x_0)} x_m^{\beta} \vert Dv\vert^2_{g}\mathrm{dvol}_{g}  \leq \varepsilon_1,
 	\end{equation*}
 	then for every $B^+_{\rho}(y)\subset B^+_R(x_0)$ with $y \in B_{\frac{R}{2}}(x_0) \cap (\mathbb{R}^{m-1}\times\{0\})$ and $\rho \leq \frac{R}{2}$ we have 
 	\begin{align}\label{impenscalbdryest}
 	&(\theta_1 \rho)^{2-m - \beta}\int_{B^+_{\theta_1 \rho}(y)} x_m^{\beta} \vert Dv\vert^2\mathrm{d}x 
   \leq  \frac{1}{2}\rho^{2-m - \beta}\int_{B^+_\rho(y)} x_m^{\beta} \vert Dv\vert^2\mathrm{d}x .
 	\end{align}
 \end{lem}

\begin{proof}
	Suppose $v$ satisfies $R^{2-m - \beta}\int_{B^+_R(x_0)} x_m^{\beta} \vert Dv\vert^2_{g}\mathrm{dvol}_{g}  \leq \varepsilon_1$  for $\varepsilon_1 >0$ to be chosen. Corollary \ref{monotonicity_inequality_Euclidean} then implies that for any $\rho  \in (0,\frac{R}{2}]$ and $y \in B_{\frac{R}{2}}(x_0) \cap (\mathbb{R}^{m-1}\times\{0\})$ we have 
	\begin{align}\label{bdryenergydecaymonbound}
	\rho^{2-m-\beta}\int_{B^+_{\rho}(y)} x_m^{\beta} \vert Dv\vert^2_{g}\mathrm{dvol}_{g} & \leq C\left(\frac{R}{2}\right)^{2-m-\beta}\int_{B^+_{\frac{R}{2}}(y)} x_m^{\beta} \vert Dv\vert^2_{g}\mathrm{dvol}_{g} \nonumber\\
	& \leq C R^{2-m-\beta}\int_{B^+_{R}(x_0)} x_m^{\beta} \vert Dv\vert^2_{g}\mathrm{dvol}_{g} \nonumber\\
	& \leq  C\varepsilon_1.
	\end{align}
    Since $v$  satisfies \eqref{wELeqnsphere} and $\vert v\vert = 1$ almost everywhere it satisfies  \eqref{wELeqnbound} with $C^{*} = 1$ for every $\psi \in C_0^{\infty}(\tilde{\mathscr{U}},\mathbb{R}^n)$. Then for every $\delta >0$ there exist constants $\varepsilon_0 = \varepsilon_0(\delta)>0$, $\theta_0 = \theta_0(\delta) \in (0,\frac{1}{8}]$ and $R_0 = R_0(\delta) \in(0,1]$ such that the conclusion of Lemma  \ref{bdrycontimppoincscal} holds. These numbers also may also depend on $m,\beta$ and other uniform constants but only the dependence on $\delta$ is important here. We henceforth assume $R_1\leq R_0$   and $\varepsilon_1 \leq C^{-1}\varepsilon_0$ where $C$ is the constant from \eqref{bdryenergydecaymonbound};  we will subsequently fix $\delta$ uniformly. In view of  \eqref{bdryenergydecaymonbound} and the choice of $\varepsilon_1$ we have $\rho^{2-m-\beta}\int_{B^+_\rho(y)} x_m^{\beta}\vert Dv\vert^2_{g}\mathrm{dvol}_{g} \leq \varepsilon_0$ and hence it follows from  Lemma  \ref{bdrycontimppoincscal} that
 	\begin{equation*}
	(\theta_0\rho)^{-(m+\beta)}\int_{B^+_{\theta_0 \rho}(y)}  x_m^{\beta} \left\vert   v - \overline{v}_{B^+_{\theta_0 \rho}(y),\beta}\right\vert^2\mathrm{d}x  \leq \delta  \rho^{2-m-\beta}\int_{B^+_\rho(y)}  x_m^{\beta} \vert Dv\vert^2_{g}\mathrm{dvol}_{g}.  
	\end{equation*}
	Let $\tilde{v} \in W^{1,2}_{\beta}(\mathscr{V};\mathbb{S}^{n-1})$ denote the even reflection of $v$ in $\mathbb{R}^{m-1}\times\{0\}$. We 
    apply  Lemma \ref{cacctypeineq} followed by Lemma \ref{bdrycontimppoincscal}, together with Corollary \ref{monotonicity_inequality_Euclidean} to see that 
	\begin{align}\label{comblemmafordecay}
	&\left(\frac{\theta_0\rho }{2}\right)^{2-m-\beta}\int_{B^+_\frac{\theta_0\rho}{2}(y)} x_m^{\beta} \vert Dv\vert^2\mathrm{d}x\nonumber\\
	&  \leq C( [\tilde{v}]_{BMO(B_{\theta_0\rho} (y))} + \delta_0) (\theta_0\rho )^{2-m-\beta}\int_{B^+_{\theta_0\rho}(y)} x_m^{\beta} \vert Dv\vert^2\mathrm{d}x \nonumber\\
	& \quad +  C\delta_0^{-1}({\theta_0\rho})^{-(m+\beta)}\int_{B^+_{\theta_0\rho}(y)} x_m^{\beta} \vert   v - \overline{v}_{B_r(y),\beta}\vert^2\mathrm{d}x\nonumber\\
	&  \leq   C([\tilde{v}]_{BMO(B_{\theta_0\rho} (y))} +  \delta_0) \rho^{2-m-\beta}\int_{B^+_\rho(y)}  x_m^{\beta} \vert Dv\vert^2\mathrm{d}x \nonumber\\
	& \quad +  C\delta_0^{-1}\delta \rho^{2-m-\beta}\int_{B^+_\rho(y)}  x_m^{\beta} \vert Dv\vert^2\mathrm{d}x\nonumber\\
	&  = C([\tilde{v}]_{BMO(B_{\theta_0\rho} (y))} +  \delta_0 + \delta\delta_0^{-1}) \rho^{2-m-\beta}\int_{B^+_\rho(y)}  x_m^{\beta} \vert Dv\vert^2\mathrm{d}x  .
	\end{align} 
	Now first choose  $\delta_0 = \frac{1}{8C}$ and then choose $\delta = \frac{1}{64 C^2}$ where $C$ is the constant from \eqref{comblemmafordecay}. This fixes $\varepsilon_1$ and $\theta_1:=\frac{\theta_0}{2} \in (0,\frac{1}{8}]$ uniformly.  It hence follows from \eqref{comblemmafordecay} that 
	\begin{align}\label{decayfixeddeltas} 
	\left(\theta_0\rho \right)^{2-m-\beta}\int_{B^+_{\theta_1\rho}(y)} x_m^{\beta} \vert Dv\vert^2\mathrm{d}x &  \leq  (C[\tilde{v}]_{BMO(B_{\theta_1\rho} (y))} + \frac{1}{4}) \rho^{2-m-\beta}\int_{B^+_\rho(y)}  x_m^{\beta} \vert Dv\vert^2\mathrm{d}x .
	\end{align} 
    Since $\theta_1\rho \leq \frac{\rho}{6}$, we may apply Lemma \ref{statharmbmo} to see that 
	\begin{align}\label{energydecaysmallconstant}
	\left(\theta_1\rho\right)^{2-m-\beta}\int_{B^+_{\theta_0\rho}(y)} x_m^{\beta} \vert Dv\vert^2\mathrm{d}x &  \leq  (\tilde{C}\varepsilon_1^{\frac{1}{2}} + \frac{1}{4}) \rho^{2-m-\beta}\int_{B^+_\rho(y)}  x_m^{\beta} \vert Dv\vert^2\mathrm{d}x.
	\end{align}
		To conclude the proof we fix $\varepsilon_0 = \min\{\frac{1}{16\tilde{C}^2}, C^{-1}\varepsilon_0\}$ where $C$ is the constant from \eqref{bdryenergydecaymonbound}.  
	 \end{proof}

 \section{$\varepsilon$-Regularity Near the Boundary and Partial Regularity}\label{epsilonreg}
 
 With the results of the previous sections in hand, we are now in a position to prove $\varepsilon$-regularity for stationary harmonic maps with free boundary data near the boundary. 
 
 \begin{thm}\label{epregthm}
 	There exists $\varepsilon_2 >0, R_2>0$ and $\gamma,\theta_2 \in (0,1)$ such that the following holds.  Suppose $v \in W^{1,2}_{\beta}(\mathscr{U};\mathbb{S}^{n-1})$ is a weakly stationary harmonic map respect to free boundary data.
  	 If $B^+_R(x_0)$ is a half-ball with $R \leq R_2$, $B_{2R}^+(x_0) \subset \mathscr{U}$ and
 	\begin{equation*}
 	R^{2-m-\beta}\int_{B^+_R(x_0)}x_{m}^{\beta}\vert Dv\vert^2_{g}\mathrm{dvol}_{g} \leq \varepsilon_2
 	\end{equation*}
 	then $v \in C^{0,\gamma}(\overline{B^+_{\theta_2 R}(x_0)};\mathbb{S}^{n-1})$. 
 \end{thm}
 \begin{proof}
 	Our goal is to apply a version of Morrey's Decay Lemma, see for example \cite{moser2005partial} Lemma 2.1 for a precise statement.  Let $\varepsilon_1 >0$, ${R}_1>0$ and $\theta_1 \in (0,\frac{1}{8}]$ be the numbers from Lemma \ref{impenscalbdry}.  Consider $B^+_{\rho}(y)$ with $\rho  \in (0,\frac{R}{2}]$ and $y \in B_{\frac{R}{2}}(x_0) \cap (\mathbb{R}^{m-1}\times\{0\})$.  If $\varepsilon_2 \leq \varepsilon_1$ then the lemma applies and iterating \eqref{impenscalbdryest} for $k \in \mathbb{N}\cup\{0\}$ yields
 	\begin{equation}\label{energydecayiterate}
 	\left(\theta_1^k\rho  \right)^{2-m-\beta}\int_{B^+_{\theta_1^k\rho}(y)} x_m^{\beta} \vert Dv\vert^2\mathrm{d}x      \leq  \frac{1}{2^k} \rho^{2-m-\beta}\int_{B^+_\rho(y)}  x_m^{\beta} \vert Dv\vert^2\mathrm{d}x.
 	\end{equation}
 	Now set $\gamma_0  = -\ln(2)(\ln(\theta_1))^{-1}\in(0,1)$. Then we have 
 	\begin{equation}\label{energydecayiterategammasub}
 	\left(\theta_1^k\rho  \right)^{2-m-\beta}\int_{B^+_{\theta_1^k\rho}(y)} x_m^{\beta} \vert Dv\vert^2\mathrm{d}x      \leq  \frac{(\theta_1^{k}\rho)^{\gamma_0}}{\rho^{\gamma_0}} \rho^{2-m-\beta}\int_{B^+_\rho(y)}  x_m^{\beta} \vert Dv\vert^2\mathrm{d}x.
 	\end{equation}
 	Now let $t \leq \rho$. Then $t \in [\theta_1^{k+1}\rho,\theta_1^{k}\rho]$ for some $k \in \mathbb{N}\cup\{0\}$ and   we see that  
 	\begin{align}\label{energydecayiterategammasubgeneralt}
 	t^{2-m-\beta}\int_{B^+_t(y)}  x_m^{\beta} \vert Dv\vert^2\mathrm{d}x & \leq \left(\theta_1^{k+1}\rho  \right)^{2-m-\beta}\int_{B_{\theta_1^{k}\rho}(y)} x_m^{\beta} \vert Dv\vert^2\mathrm{d}x\nonumber\\
 	& = \theta_1^{2-m-\beta}  \left(\theta_1^{k}\rho  \right)^{2-m-\beta}\int_{B^+_{\theta_1^{k}\rho}(y)} x_m^{\beta} \vert Dv\vert^2\mathrm{d}x\nonumber\\
 	&     \leq  C\frac{(\theta_1^{k}\rho)^{\gamma_0}}{\rho^{\gamma_0}} \rho^{2-m-\beta}\int_{B^+_\rho(y)}  x_m^{\beta} \vert Dv\vert^2\mathrm{d}x\nonumber\\
 	& \leq    C \left(\frac{t}{\rho}\right)^{\gamma_0} \rho^{2-m-\beta}\int_{B^+_\rho(y)}  x_m^{\beta} \vert Dv\vert^2\mathrm{d}x.
 	\end{align}
 	Now we consider balls in $B_\rho(y) \subset B^+_R(x_0)$ with $\text{dist}(B_\rho(y);\mathbb{R}^{m-1}\times\{0\}) \geq \rho$ and $y \in B^+_{\frac{R}{2}}(x_0)$. In this case, the
Riemannian metric, restricted to $B_\rho(y)$ and multiplied by $\rho^{-\alpha}$, is
uniformly bounded from above and below. We have similar control for all
the derivatives of this multiple of the metric in $B_\rho(y)$.
It hence follows from the monotonicity inequality, Corollary \ref{monotonicity_inequality_Euclidean},
and from \cite{MR1208652} and \cite{MR765241} that there is a constant $C$ and a $\gamma_1$ which do not depend on $\rho$ or $y$  such that for $t \leq \rho$
    \begin{align}\label{energydecayinterior}
    t^{2-m}\int_{B_t(y)}  \vert Dv\vert^2\mathrm{d}x  
    & \leq    C \left(\frac{t}{\rho}\right)^{\gamma_0} \rho^{2-m}\int_{B_\rho(y)}  \vert Dv\vert^2\mathrm{d}x.
    \end{align}
 	Setting $\gamma = \min\{\gamma_0,\gamma_1\}$ the remainder of the proof follows exactly as in the proof of Theorem 4.21 of \cite{roberts2018regularity}; we combine \eqref{energydecayiterategammasubgeneralt}, \eqref{energydecayinterior} and an application of  a version of Morrey's Decay Lemma adapted to the situation considered here, see \cite{roberts2018regularity} Lemma 4.8,  to conclude  the proof. 
 	 \end{proof}
 
 From $\varepsilon$-regularity we deduce partial regularity using a covering argument.   
 Let $\mathscr{H}^t$ denote the $t$-dimensional Hausdorff measure. The following is the full version
 of Theorem \ref{mainthmabridged}.
 
 \begin{thm}\label{partialregbdrycoordchart}
Let $\beta \in (-1,1)$. 	Suppose $v \in W^{1,2}_{\beta}(\mathcal{M};\mathbb{S}^{n-1})$ is weakly stationary harmonic with respect to free
boundary data.
There exists a relatively closed set $\Sigma\subset\mathcal{M}$ with $\mathscr{H}^{m-2+\beta}(\Sigma\cap \mathrm{int} \mathcal{M}) = 0$ and $\mathscr{H}^{m-2}(\Sigma\cap\partial \mathcal{M}) = 0$  and there exists $\gamma \in (0, 1)$
such that $v \in C^{0,\gamma}_{loc}(\mathcal{M} \setminus \Sigma;\mathcal{N})$.
 \end{thm}
 
 \begin{proof}
It suffices to prove the statement in coordinate patches. Moreover, in the interior of $\mathcal{M}$,
the statement follows from the known regularity theory for harmonic maps, in particular \cite{MR1143435}.
Hence we may consider $\mathscr{U} \subset \mathbb{R}^{m - 1} \times [0, \infty)$ as in Section \ref{metriccoord} and regard $g$ as a Riemannian metric on $\mathscr{U}$,
as we did in the last few sections. 

Define
\begin{equation*}
\Sigma_{I} = \{x \in \mathscr{U} \cap (\mathbb{R}^{m - 1} \times (0, \infty)): v\ \text{is not smooth in any neighbourhood of}\ x\}. 
\end{equation*}
By definition $\mathscr{U}\backslash \Sigma_{I}$ is relatively
open and the partial regularity theory of Evans \cite{MR1143435}, implies that    $\mathscr{H}^{m-2}(\Sigma_I) = 0$.
 
  Now define
\begin{align*}
\Sigma_{B} = \{x \in \mathscr{U}\cap(\mathbb{R}^{m-1}\times\{0\}): \lim_{\rho \to 0^{+}}\rho^{2-m-\beta}\int_{B^+_{\rho}(y)}x_m^{\beta}\vert Dv\vert^2\mathrm{d}x >0\} 
\end{align*} 
and let $\Sigma = \Sigma_{B} \cup\Sigma_{I}$. Let $\partial^0\mathscr{U}:= \mathscr{U}\cap (\mathbb{R}^{m-1}\times\{0\})$.
 We show $\mathscr{U}\backslash \Sigma = \mathscr{U}\backslash\Sigma_{I} \cup  \partial^0\mathscr{U}\backslash \Sigma_{B}$ is relatively open in $\mathscr{U}$. Let $x \in  \mathscr{U}\backslash \Sigma$. If $x \in \mathbb{R}^{m - 1} \times (0, \infty)$  then by definition $v$ is smooth in a neighbourhood of $x$ and this neighbourhood is therefore contained in $\mathscr{U}\backslash\Sigma_{I}\subset \mathscr{U}\backslash \Sigma$. If $x \in \partial^0\mathscr{U}\backslash \Sigma_{B}$  then there exists $R>0$ such that $R^{2-m-\beta}\int_{B^+_{\rho}(x)}y_m^{\beta}\vert Dv\vert^2_{g}\mathrm{dvol}_{g} \leq \varepsilon_2$, where $\varepsilon_2$ is the number from Theorem \ref{epregthm}. Hence there exists $\theta_2,\gamma \in  (0,1)$  such that $v \in C^{0,\gamma}(\overline{B^+_{\theta_2 R}(x_0)};\mathbb{S}^{n-1})$. It then follows from regularity theory for harmonic maps, see e.g \cite{MR765241} Lemma 3.1, that since $v$ is continuous in $B^+_{\theta_2 R}(x_0)$ it is smooth there and hence $v$ is smooth in a neighbourhood of any point in $B^+_{\theta_2 R}(x_0)$. Furthermore, it follows from the proof of the theorem, see \eqref{energydecayiterategammasubgeneralt}, that  
 \begin{align}\label{energydecayiterategammasubgeneraltrec}
 t^{2-m-\beta}\int_{B^+_t(y)}  x_m^{\beta} \vert Dv\vert^2\mathrm{d}x & \leq
  C \left(\frac{t}{\rho}\right)^{\gamma_0} \rho^{2-m-\beta}\int_{B^+_\rho(y)}  x_m^{\beta} \vert Dv\vert^2\mathrm{d}x 
 \end{align}
 for every  $B^+_{\rho}(y)$ with $\rho  \in (0,\frac{R}{2}]$ and $y \in B_{\frac{R}{2}}(x_0) \cap (\mathbb{R}^{m-1}\times\{0\})$ and every $t \leq \rho$. Fixing $\rho = \frac{R}{2}$ we see that 
  \begin{align}\label{energydecayiterategammasubgeneraltpartialregapp}
 t^{2-m-\beta}\int_{B^+_t(y)}  x_m^{\beta} \vert Dv\vert^2\mathrm{d}x & \leq
 C \left(\frac{2t}{R}\right)^{\gamma_0} \left(\frac{R}{2}\right)^{2-m-\beta}\int_{B^+_\frac{R}{2}(y)}  x_m^{\beta} \vert Dv\vert^2\mathrm{d}x \nonumber\\
 & \leq C \left(\frac{t}{R}\right)^{\gamma_0} \varepsilon_2
 \end{align}
 and hence every $y \in B_{\frac{R}{2}}(x_0) \cap (\mathbb{R}^{m-1}\times\{0\})$ belongs to $\partial^0\mathscr{U}$. Setting $\sigma = \min\{\frac{1}{2},\theta_2\}$ we then see that $\overline{B^+_{\sigma R}(x_0)}\subset \mathscr{U}\backslash\Sigma$. Hence $\Sigma$ is relatively closed. We see that $\mathscr{H}^{m-2+\beta}(\Sigma_B) = 0$ using a covering argument analogous to that of the proof of  \cite{roberts2018regularity} Theorem 4.3.
 \end{proof}

\subsubsection*{Acknowledgements}
The second author was funded by the Royal Society Grant with Grant Number: RP\textbackslash R1\textbackslash180114.

 \bibliographystyle{plain}
 \bibliography{Stationaryharmonic}

\end{document}